\documentclass{paper}

\usepackage{graphicx}
\usepackage{amsmath,epsfig,amssymb,mathrsfs}
\usepackage{algorithmic}
\usepackage{cite}
\usepackage{amsthm}
\usepackage{arydshln}
\usepackage{amsopn}
\usepackage{color}
\usepackage{verbatim}
\usepackage{amsopn}
\usepackage{amssymb}
\usepackage{algorithm}

\newtheorem{theorem}{Theorem}[section]
\newtheorem{lemma}[theorem]{Lemma}
\newtheorem{definition}[theorem]{Definition}
\newtheorem{proposition}[theorem]{Proposition}
\newtheorem{corollary}[theorem]{Corollary}

\newcommand{\bE}{\mathbb{E}}

\newcommand{\bI}{\mathbb{I}}

\newcommand{\bP}{\mathbb{P}}
\newcommand{\hsig}{\hat{\sigma}}

\begin{document}

\title{Asymptotic Analysis of LASSO's Solution Path with Implications for Approximate Message Passing}
\author{Ali Mousavi, Arian Maleki, Richard G. Baraniuk}
\maketitle
\begin{abstract}
This paper concerns the performance of the LASSO (also knows as basis pursuit denoising) for recovering sparse signals from undersampled, randomized, noisy measurements.  
We consider the recovery of the signal $x_o \in \mathbb{R}^N$ from $n$ random and noisy linear observations $y= Ax_o + w$, where $A$ is the measurement matrix and $w$ is the noise. 
The LASSO estimate is given by the solution to the optimization problem $x_o$ with $\hat{x}_{\lambda} = \arg \min_x \frac{1}{2} \|y-Ax\|_2^2 + \lambda \|x\|_1$. 
Despite major progress in the theoretical analysis of the LASSO solution, little is known about its behavior as a function of the regularization parameter $\lambda$.  
In this paper we study two questions in the asymptotic setting (i.e., where $N \rightarrow \infty$, $n \rightarrow \infty$ while the ratio $n/N$ converges to a fixed number in $(0,1)$):   (i) How does the size of the active set $\|\hat{x}_\lambda\|_0/N$ behave as a function of $\lambda$, and (ii) How does the mean square error $\|\hat{x}_{\lambda} - x_o\|_2^2/N$ behave as a function of $\lambda$? 
We then employ these results in a new, reliable algorithm for solving LASSO based on approximate message passing (AMP). 
\end{abstract}

\section{Introduction}\label{sec:intro}


\subsection{Motivation}
Consider the problem of recovering a vector $x_o \in \mathbb{R}^N$ from a set of undersampled random linear measurements $y= Ax_o+w$, where $A \in \mathbb{R}^{n \times N}$ is the measurement matrix, and $w \in \mathbb{R}^n$ denotes the noise. One of the most successful recovery algorithms, called basis pursuit denoising or LASSO (\cite{Tiblasso96, ChDoSa98}), that employs the following optimization problem to obtain an estimate of $x_o$:

\begin{equation}\label{eq:LASSO}
\hat{x}_{\lambda} = \arg \min_x \frac{1}{2}\|y- Ax\|_2^2 + \lambda \|x\|_1.   
\end{equation}
A rich literature has provided a detailed analysis of this algorithm \cite{DET, Tropp, ZhYu06, MeYu09, BiRiTs08, MeBu06, GeBu09, BuTsWe07, KnFu2000, ZoHaTib2007, DoTa05, Do05, DoTa08, DoTa09, DoTa09b, MalekiThesis, BaMo10, BaMo11,amelunxen2013living,oymak2012relation}.  Most of the work published in this area falls into two categories: (i) non-asymptotic and (ii) asymptotic results. The non-asymptotic results consider $N$ and $n$ to be large but finite numbers and characterize the reconstruction error as a function of $N$ and $n$. These analyses provide qualitative guidelines on how to design compressed sensing (CS) system. However, they suffer from loose constants and are incapable of providing quantitative guidelines. Therefore, inspired by the seminal work of Donoho and Tanner \cite{DoTa05}, researchers have started the asymptotic analysis of LASSO. Such analyses provide sharp quantitative guidelines for designing CS systems. 

Despite the major progress in our understanding of LASSO, one major aspect of the method that is of major algorithmic importance has remained unexplored. In most of the theoretical work, it is assumed that an oracle has given the optimal value of $\lambda$ to the statistician/engineer and the analysis is performed for the optimal value of $\lambda$. However, in practice the optimal value of $\lambda$ is not known a priori. One important analysis that may help in both searching for the optimal value of $\lambda$ and/or designing efficient algorithms for solving LASSO, is the behavior of the solution $\hat{x}_{\lambda} $ as a function of $\lambda$. In this paper, we conduct such an analysis and demonstrate how such results can be employed for designing efficient approximate message passing algorithms.

\subsection{Analysis of LASSO's solution path}

In this paper we aim to analyze the properties of the solution of the LASSO as $\lambda$ changes. The two main problems that we address are:

\begin{itemize}
\item Q1: How does $\frac{1}N{}\|\hat{x}_{\lambda}\|_0$ change as $\lambda$ varies?
\item Q2: How does $\frac{1}{N}\|\hat{x}_\lambda  - x\|_2^2$ change as $\lambda$ varies?
\end{itemize}

The first question is about the number of active elements in the solution of the LASSO, and the second one is about the mean squared error. 
Intuitively speaking, one would expect the size of the active set to shrink as $\lambda$ increases and the mean square error to be a bowl-shaped function of $\lambda$. Unfortunately the peculiar behavior of LASSO breaks this intuition. See Figure \ref{fig:activeset} for a counter-example; we will clarify the details of this experiment in Section \ref{sec:simulationdetails}. This figure exhibits the number of active elements in the solution as we increase the value of $\lambda$. It is clear that the size of the active set is not monotonically decreasing.  

\begin{figure}[h!]
\includegraphics[width= 12cm]{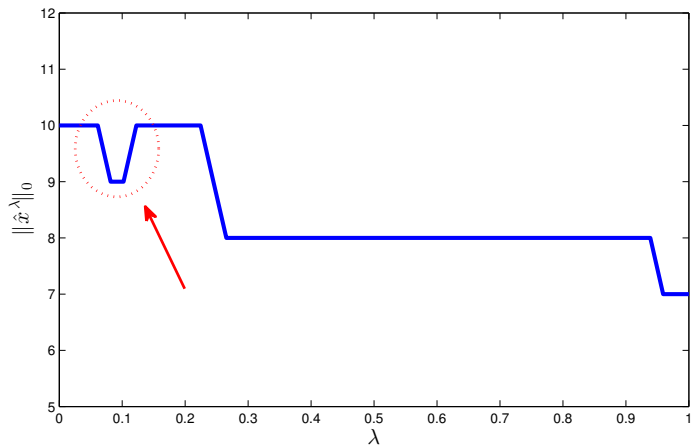}
\caption{The number of active elements in the solution of LASSO as a function of $\lambda$. It is clear that this function does not match the intuition. The size of the active set at one location grows as we increase $\lambda$. See the details of this experiment in Section \ref{sec:simulationdetails}. }
\label{fig:activeset}
\end{figure}

 Such pathological examples have discouraged further investigation of these problems in the literature. The main objective of this paper is to show that such pathological examples are quite rare, and if we consider the asymptotic setting (that will be described in Section \ref{sec:asympframework}), then we can provide quite intuitive answers to the two questions raised above. Let us summarize our results here in a non-rigorous way. We will formalize these statements and clarify the conditions under which they hold in Section \ref{sec:lassopath}. 
 
 \begin{itemize}
\item A1: In the asymptotic setting, $\frac{1}N{}\|\hat{x}_{\lambda}\|_0$ is a decreasing function of $\lambda$. 
\item A2: In the asymptotic setting, $\frac{1}{N}\|\hat{x}_\lambda  - x\|_2$ is a quasi-convex function of $\lambda$. 
\end{itemize}

 \subsection{Implications for approximate message passing algorithms}\label{sec:ampintro}
 Traditional techniques of solving LASSO, such as the interior point method, have fail in addressing high-dimensional CS-type problems. Therefore, researchers have started exploring iterative algorithms with inexpensive per-iteration computations. One such algorithm is called {\em approximate message passing} (AMP) \cite{DoMaMo09}; it is given by the following iteration:
\begin{eqnarray}\label{eq:ampeq1}
x^{t+1} &=& \eta(x^t + A^* z^t; \tau^t), \nonumber \\
z^t &=& y- Ax^t + \frac{|I^t|}{n} z^{t-1}.
\end{eqnarray}

AMP is an iterative algorithm, and $t$ is the index of iteration. $x^t$ is the estimate of $x_o$ at iteration $t$. $\eta$ is the soft thresholding function applied component-wise to the elements of the vector. For $a \in \mathbb{R}$, $\eta (a ; \tau) \triangleq (|a| - \tau)_{+} {\rm sign}(a)$. $I^t \triangleq \{i \ : \  x_i^t \neq 0  \}$. Finally $\tau^t$ is called the threshold parameter. One of the most interesting features of AMP is that, in the asymptotic setting (which will be clarified later), the distribution of $v^t \triangleq x^t+A^* z^t - x_o$ is Gaussian at every iteration, and it can be considered to be independent of $x_o$.  Figure \ref{fig:GammaFunc} shows the empirical distribution of $v^t$ at a three different iterations.

\begin{figure}
\includegraphics[width= 11cm]{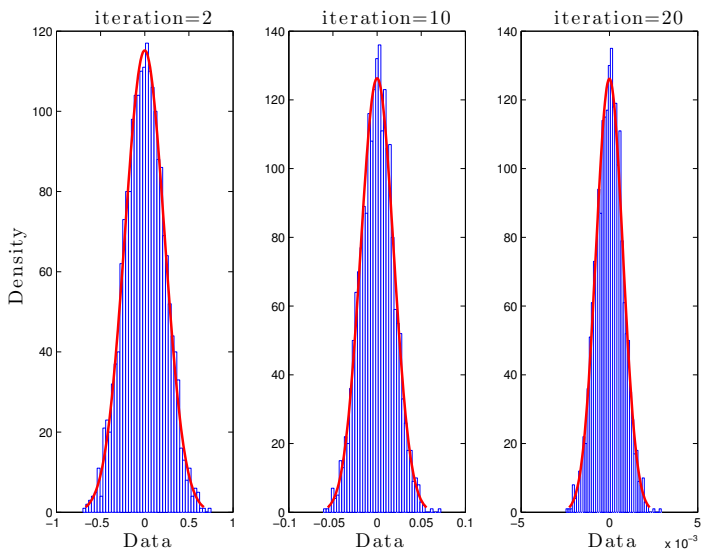}
\caption{Histogram of $v^t$ for three different iterations. The red curve displays the best Gaussian fit.}
\label{fig:GammaFunc}
\end{figure}

As is clear from \eqref{eq:ampeq1}, the only parameter that exists in the AMP algorithm is the threshold parameter $\tau^t$ at different iterations. It turns out that different choices of this parameter can lead to very different performance. One choice that has interesting theoretical properties was first introduced in \cite{DoMaMo09, DoMaMoNSPT} and is based on the Gaussianity of $v^t$. Suppose that an oracle gives us the standard deviation of $v^t$ at time $t$, called $\sigma^t$. Then one way for determining the threshold is to set $\tau^t = \beta \sigma^t$, where $\beta$ is a fixed number. This is called the {\em fixed false alarm thresholding policy}. It turns out that if we set $\beta$ properly in terms of $\lambda$ (the regularization parameter of LASSO), then $x^t$ will eventually converge to $\hat{x}^{\lambda}$. The nice theoretical properties of the fixed false alarm thresholding policy come at a price, however,  and that is the requirement for estimating $\sigma^t$ at every iteration, which is not straightforward as we observe $x_o+v^t$ and not $v^t$. However, the fact that the size of the active set of LASSO is a monotonic function of $\lambda$ provides a practical and easy way for setting $\tau^t$. We call this approach {\em fixed detection thresholding}. 

  \begin{definition}[Fixed detection thresholding policy]
Let $ 0 < \gamma <1$. Set the threshold value, $\tau^t$ to the absolute value of the $\lfloor \gamma n \rfloor^{\rm th}$ largest element (in absolute value) of $x^t+ A^T z^t$.
\end{definition}

Note that a similar thresholding policy has been employed for iterative hard thresholding \cite{BlDaCS09, BlDa09}, iterative soft thresholding \cite{Maleki09}, and AMP \cite{MaMoCISS10} in a slightly different way. In these works, it is assumed that the signal is sparse and its sparsity level is known, and $\gamma$ is set according to the sparsity level. However, here $\gamma$ is assumed to be a free parameter. In the asymptotic setting, AMP with this thresholding policy is also equivalent to the LASSO in the following sense: for every $\lambda>0$ there exists a unique $\gamma \in (0,1)$ for which AMP converges to the solution of LASSO as $t \rightarrow \infty$. This result is a conclusion of the monotonicity of the size of the active set of LASSO in terms of $\lambda$. We will formally state our results regarding the AMP algorithm with fixed detection thresholding policy in Section \ref{sec:ampimplic}. 

\subsection{Organization of the paper}

The organization of the paper is as follows: Section \ref{sec:maincontribution} sets up the framework and formally states the main contributions of the paper. Section \ref{sec:Thms} proves the main results of this paper. Section \ref{sec:simulations} summarizes our simulation results and finally, Section \ref{sec:con} includes the conclusion of the paper.
\section{Main contributions}\label{sec:maincontribution}

\subsection{Notation}
 Capital letters denote both matrices and random variables. As we may consider a sequence of vectors with different sizes, sometimes we denote $x$ with $x(N)$ to emphasize its dependency on the ambient dimension. For a matrix $A$, $A^T$, $\sigma_{\rm min} (A)$, and $\sigma_{\max} (A)$ denote the transpose of $A$, the minimum, and maximum singular values of $A$ respectively. Calligraphic letters such as $\mathcal{A}$ denote sets. For set $\mathcal{A}$, $|\mathcal{A}|$, and $|\mathcal{A}^c|$ are the size of the set and its complement respectively. For a vector $x \in \mathbb{R}^n$, $x_i$, $\| x\|_p \triangleq (\sum |x_i|^p)^{1/p}$, and $\|x\|_0 = |\{i \ : \  |x_i| \neq 0 \} |$ represent the $i^{\rm th}$ component, $\ell_p$, and $\ell_0$ norms respectively. We use $\mathbb{P}$ and $\mathbb{E}$ to denote the probability and expected value with respect to the measure that will be clear from the context. The notation $\mathbb{E}_X $ denotes the expected value with respect to the randomness in random variable $X$. The two functions $\phi$ and $\Phi$ denote the probability density function and cumulative distribution function of standard normal distribution. Finally, $\mathbb{I}(\cdot)$ and ${\rm sign(\cdot)}$ denote the indicator and sign functions, respectively.

\subsection{Asymptotic CS framework}\label{sec:asympframework}
In this paper we consider the problem of recovering an approximately sparse vector $x_o \in \mathbb{R}^N$  from $n$ noisy linear observations $y = A x_o+ w$.  
Our main goal is to analyze the properties of the solution of LASSO, defined in \eqref{eq:LASSO}, on CS problems with the following two main feature. (i) the measurement matrix has iid gaussian elements,\footnote{With the recent results in CS \cite{BaMoLe12} our results can be easily extended to subgaussian matrices. However, for notational simplicity we consider the Gaussian setting here. } and  (ii) the ambient dimension and the number of measurements are large. We adopt  the {\em asymptotic framework} to incorporate these two features. Here is the formal definition of this framework \cite{DoMaMoNSPT, BaMo10}. Let $n,N \rightarrow \infty$ while $\delta = \frac{n}{N}$ is fixed. We write the vectors and matrices as $x_o(N), A(N), y(N)$, and $w(N)$ to emphasize on the ambient dimension of the problem. Clearly, the number of row of the matrix $A$ is equal to $\delta N$, but we assume that $\delta$ is fixed and therefore we do not include $n$ in our notation for $A$.  The same argument is applied to $y(N)$ and $w(N)$. 
 
\begin{definition}\label{def:convseq}
A sequence of instances $\{x_o(N), A(N), w(N)\}$ is called a converging sequence if the following conditions hold:
\begin{itemize}
\item[-] The empirical distribution of $x_o(N) \in \mathbb{R}^N$ converges weakly to a probability measure $p_{X}$ with bounded second moment.
\item[-] The empirical distribution of $w(N) \in \mathbb{R}^n$ ($n = \delta N$) converges weakly to a probability measure $p_W$ with bounded second moment.
\item[-] If $\{e_i\}_{i=1}^N$ denotes the standard basis for $\mathbb{R}^N$, then \newline $\max_i \|A(N) e_i \|_2, \min_i \|A(N) e_i \|_2 \rightarrow 1$ as $N \rightarrow \infty$. 
\end{itemize}
\end{definition}

Note that we have not imposed any constraint on the limiting distributions $p_X$ or $p_W$. In fact for the purpose of this section, $p_X$ is not necessarily a sparsity promoting prior. Furthermore, unlike most of the other works that assumes $p_W$ is Gaussian, we do not even impose this constraint on the noise. Also, the last condition is equivalent to saying that all the columns have asymptotically unit $\ell_2$ norm. For each problem instance $x_o(N), A(N),$ and $w(N)$ we solve LASSO and obtain $\hat{x}^{\lambda}(N)$ as the estimate. We would now like to evaluate certain measures of performance for this estimate such as the mean squared error. The next generalization formalizes the types of measure we are interested in. 

\begin{definition}\label{def:observablelasso}
Let $\hat{x}^{\lambda}(N)$ be the sequence of solutions of the LASSO problem for the converging sequence of instances $\{x_o(N), A(N), w(N)\}$. Consider a function $\psi: \mathbb{R}^2 \rightarrow \mathbb{R}$. An observable $J_{\psi}$ is defined as 
\[
J_{\psi} \left(x_o,\hat{x}^{\lambda}\right) \triangleq \lim_{N \rightarrow \infty} \frac{1}{N} \sum_{i=1}^N \psi \left( x_{o,i}(N), \hat{x}^{\lambda}_i(N)\right).
\]
\end{definition}
A popular choice of the $\psi$ function is $\psi_M(u,v)= (u-v)^2$. For this function the observable has the form:
\[
J_{\psi_M} \left(x_o,\hat{x}^{\lambda}\right) \triangleq \lim_{N \rightarrow \infty} \frac{1}{N} \sum_{i=1}^N  \left( x_{o,i}(N)- \hat{x}^{\lambda}_i(N)\right)^2 =  \lim_{N \rightarrow \infty}  \frac{1}{N} \|x_o - \hat{x}^\lambda\|_2^2. 
\]
Another example of $\psi$ function that we consider in this paper is $\psi_D (u,v) = \mathbb{I} (v \neq 0)$, which leads us to
\begin{equation}\label{eq:detobs}
J_{\psi_D} \left(x_o,\hat{x}^{\lambda}\right) \triangleq \lim_{N \rightarrow \infty} \frac{1}{N} \sum_{i=1}^N \mathbb{I} (\hat{x}^{\lambda}_i \neq 0) = \lim_{N \rightarrow \infty} \frac{\|\hat{x}^{\lambda}\|_0}{N}. 
\end{equation}

Some of the popular observables are summarized in Table \ref{table:observables} with their corresponding $\psi$ functions. Note that so far we do not have any major assumption on the sequences of matrices. Following the other works in CS, we would now consider random measurement matrices. While all our discussion can be extended to more general classes of random matrices \cite{BaMoLe12}, for the notational simplicity we consider $A_{ij} \sim N(0,1/n)$. Clearly, these matrices satisfy the unit norm column condition of converging sequences with high probability. Since $\hat{x}^{\lambda}(N)$ is random, there are two questions that need to be addressed about $\lim_{N \rightarrow \infty} \frac{1}{N} \sum_{i=1}^N \psi \left( x_{o,i}(N), \hat{x}^{\lambda}_i(N)\right)$. (i) Does it exist and in what sense (e.g., in probability or almost surely)? (ii) Does it converge to a random variable or to a deterministic quantity?  The following theorem, conjectured in \cite{DoMaMoNSPT} and proved in \cite{BaMo11}, shows that under some restrictions on the $\psi$ function, not only the almost sure limit exists in this scenario, but also it converges to a non-random number.

\begin{table}
\begin{center}
\caption{Some observables and their abbreviations. The $\psi$ function for each observable is also specified.}
\begin{tabular}{|l|l|l|}
\hline
Name   & Abbreviation & $\psi=\zeta(u,w)$ \\
\hline
Mean Square Error  &MSE   & $\psi = (u-w)^2$     \\
False Alarm Rate     & FA         & $\psi = \mathbb{I}({  w \neq 0 , u =  0})$  \\
Detection Rate          & DR                & $\psi =\mathbb{I}({  w \neq 0})$ \\
Missed Detection &MD & $\psi = \mathbb{I}({ w =  0, u \neq 0  })$ \\
\hline
\end{tabular}
\label{table:observables}
\end{center}
\end{table}

\begin{theorem}\label{thm:eqpseudolip} Consider a converging sequence $\{x_o(N), A(N), w(N)\}$ and let the elements of $A$ be drawn iid from $N(0,1/n)$. Suppose that $\hat{x}^{\lambda}(N)$ is the solution of the LASSO problem. 
Then for any pseudo-Lipschitz function\footnote{A function $\psi:\mathbb{R}^2 \rightarrow \mathbb{R}$ is pseudo-Lipschitz if there exists a constant $L>0$ such that for all $x,y \in \mathbb{R}^2$ we have $|\psi(x)-\psi(y)|\leq L(1+\|x\|_2+\|y\|_2)\|x-y\|_2$.} $\psi: \mathbb{R}^2 \rightarrow \mathbb{R}$, almost surely
\begin{equation}\label{eq:lassoobs}
\lim_{N \rightarrow \infty} \frac{1}{N} \sum_i \psi \left(\hat{x}^{\lambda}_i(N),{x}_{o,i} \right) = \bE_{X_o,W} [\psi(\eta(X_o+\hsig W; \beta\hsig), X_o)],
\end{equation}
where on the right hand side $X_o$ and $W$ are two random variables with distributions $p_X$ and $N(0, 1)$, respectively, $\eta$ is the soft thresholding operator, and $\hsig$ and $\beta$ satisfy the following equations:
\begin{eqnarray} \label{eq:fixedpoint}
\hsig^2 &=& \sigma_{\omega}^2+\frac{1}{\delta} \mathbb{E}_{X, W} [(\eta(X +\hsig W; \beta\hsig) -X)^2], \label{eq:fixedpoint11} \\
\lambda &=& \beta\hsig \left(1-\frac{1}{\delta} \mathbb{P}(|X +\hsig W| > \beta\hsig) \right). \label{eq:fixedpoint21}
\end{eqnarray}
\end{theorem}

This theorem will provide the first step in our analysis of the LASSO's solution path. Before we proceed to the implications of this theorem, let us explain some of its interesting features. Suppose that $\hat{x}^{\lambda}$ has iid elements, and each element is in law equal to $\eta(X_{o} +\hsig W; \beta\hsig)$, where $X_{o}{\sim } p_X$ and $W {\sim} N(0, 1)$. Also, $x_{o,i} \overset{iid}{\sim} p_X$. If these two assumptions were true, then we could use strong law of large numbers (SLLN) and argue that \eqref{eq:lassoobs} were true under some mild conditions (required for SLLN). While this heuristic is not quite correct, and the elements of $\hat{x}^{\lambda}_i$ are not necessarily independent, at the level of calculating observables defined in Definition \ref{def:observablelasso} (and $\psi$ pseudo Lipschitz) this theorem confirms the heuristic. Note that the key element that has led to this heuristic is the randomness in the measurement matrix and the large size of the problem. \\

As we see in \eqref{eq:lassoobs}, there are two constants, $(\beta,\hsig)$, that are calculated according to \eqref{eq:fixedpoint11} and \eqref{eq:fixedpoint21}. \cite{DoMaMoNSPT, DoMaMo09} have shown that for a fixed $\lambda$, these two equations have a unique solution for $(\beta,\hsig)$. Note that here $\hsig \geq \sigma_w$, i.e., the variance of the noise that we observe after the reconstruction, $\hsig$, is larger than the input noise (according to \eqref{eq:fixedpoint11}). The extra noise that we observe after the reconstruction is due to subsampling. In fact, if we keep $\beta$ fixed and decrease $\delta$, then we see that $\hsig$ increases.  This phenomena is sometimes called noise-folding in the CS literature \cite{davenport2012pros,treichler2009application}. 


One of the main applications of Theorem \ref{thm:eqpseudolip} is in characterizing the normalized mean squared error of the LASSO problem as is summarized by the next corollary.

\begin{corollary}\label{cor:lassomsese}
If  $\{x_o(N), A(N), w(N)\}$  is a converging sequence and $\hat{x}_{\lambda}(N)$ is the solution of the LASSO problem, then almost surely
\[
\lim_{N \rightarrow \infty} \frac{1}{N}\| \hat{x}^{\lambda}(N) - x_o(N)\|_2^2 = E_{X_o,W} \left[|\eta(X_o+\hsig W; \beta\hsig)-X_o|^2\right],
\]
where  $\beta$, $\hsig$, and $\lambda$ satisfy \eqref{eq:fixedpoint11} and \eqref{eq:fixedpoint21}.
\end{corollary}

As we mentioned before, we are also interested in another observable and that is $\lim_{N \rightarrow \infty} \frac{\|\hat{x}_{\lambda} \|_0}{N}$. As described in \eqref{eq:detobs}, this observable can be constructed by using $\psi(u,v) = \mathbb{I} (v \neq 0)$. However, it is not difficult to see that for this observable, the $\psi$ function is not pseudo-Lipschitz, and hence Theorem \ref{thm:eqpseudolip} does not apply.  However, as conjectured in \cite{DoMaMoNSPT} and proved in \cite{BaMo11} we can still characterize the almost sure limit of this observable. \\

\begin{theorem} {\rm \cite{BaMo11}} \label{thm:lassodetstate}
If  $\{x_o(N), A(N), w(N)\}$  is a converging sequence and $\hat{x}^{\lambda}(N)$ is the solution of the LASSO problem, then almost surely 
\[
\lim_{N \rightarrow \infty} \frac{1}{N} \sum_i \mathbb{I} \left(\hat{x}^{\lambda}_i(N) \neq 0 \right)  = \mathbb{P} (| \eta(X_o+\hsig W; \tau\hsig)|>0),
\]
where  $\lambda$, $\tau$, and $\hsig$ satisfy \eqref{eq:fixedpoint11} and \eqref{eq:fixedpoint21}.
\end{theorem}

\subsection{LASSO's solution path}\label{sec:lassopath}
In Section \ref{sec:asympframework} we characterized two simple expressions for the asymptotic behavior of normalized mean square error and normalized number of detections. These two expressions enable us to formalize the two questions that we raised in the Introduction. 

As mentioned in the Introduction, if we consider a generic CS problem, there are some pathological examples for which the behavior of LASSO is quite unpredictable and inconsistent with our intuition. See Figure \ref{fig:activeset} for an example and Section \ref{subsec:fig:activeset} for a detailed description about it. Here, we consider the asymptotic regime introduced in the last section. It turns out that in this setting the solution of LASSO behaves as expected.

\begin{theorem}\label{lem:activeset}
Let  $\{x_o(N), A(N), w(N)\}$  denotes a converging sequence of problem instances as defined in \ref{def:convseq}. Suppose that $A_{ij}\overset{iid}{\sim} N(0, 1/n)$. If $\hat{x}^{\lambda}(N)$ is the solution of LASSO with regularization parameter $\lambda$, then
\[
\frac{d}{d \lambda} \left(\lim_{N \rightarrow \infty} \frac{1}{N} \sum_i \mathbb{I} \left(\hat{x}^{\lambda}_i(N) \neq 0 \right) \right) <0. 
\]
Furthermore, $\lim_{N \rightarrow \infty} \frac{1}{N} \sum_i \mathbb{I} \left(x^{\lambda}_i(N) \neq 0 \right) \leq \delta$.
\end{theorem}
\noindent We summarize the proof of this theorem in Section \ref{sec:proofactiveset}. \\

 Intuitively speaking, Theorem \ref{lem:activeset} claims that, as we increase the regularization parameter $\lambda$, the number of elements in the active set is decreasing. Also, according to the condition $\lim_{N \rightarrow \infty} \frac{1}{N} \sum_i \mathbb{I} \left(x^{\lambda}_i(N) \neq 0 \right) \leq \delta$ the largest it can get is $\delta = n/N$. Since the number of active elements is a decreasing function of $\lambda$, $\delta$ appears only in the limit $\lambda \rightarrow 0$. Figure \ref{fig:LassoPathRandom} shows the number of active elements as a function of $\lambda$ for a setting described in Section \ref{subsec:fig:LassoPathRandom}. In the next section, we will exploit this property to design and tune AMP for solving the LASSO. 

\begin{figure}[h!]
\includegraphics[width= 12cm]{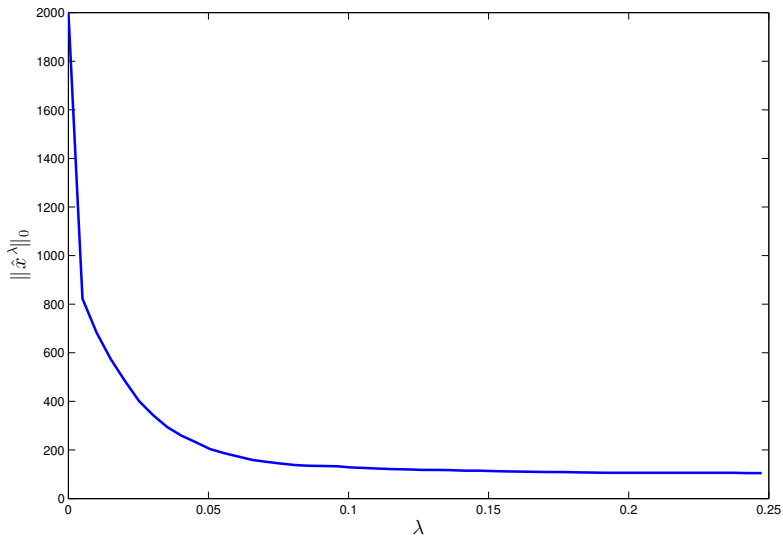}
\caption{The number of active elements in the solution of LASSO as a function of $\lambda$. The size of the active set decreases monotonically as we increase $\lambda$.}
\label{fig:LassoPathRandom}
\end{figure}

Our next result is regarding the behavior of the normalized MSE in terms of the regularization parameter $\lambda$. In asymptotic setting, we prove that the normalized MSE is a quasi-convex function of $\lambda$. See Section 3.4 of \cite{BoydVanderberghe} for a short introduction on quasi-convex functions. Figure \ref{fig:MSE} exhibits the behavior of MSE as a function of $\lambda$. The detailed description of this problem instance can be found in Section \ref{sec:simulationdetails}. Before we proceed further, we define bowl-shaped functions.

\begin{definition}\label{def:Qcvx}
A quasi-convex function $f:\mathbb{R} \rightarrow \mathbb{R}$  is called bowl-shaped if and only if there exists $x_o \in \mathbb{R}$ at which $f$ achieves its minimum, i.e.,
\[
f(x_0) \leq f(x), \ \  \forall x \in \mathbb{R}.
\]
\end{definition}

Here is the formal statement of this result.

\begin{figure}
\includegraphics[width= 12cm]{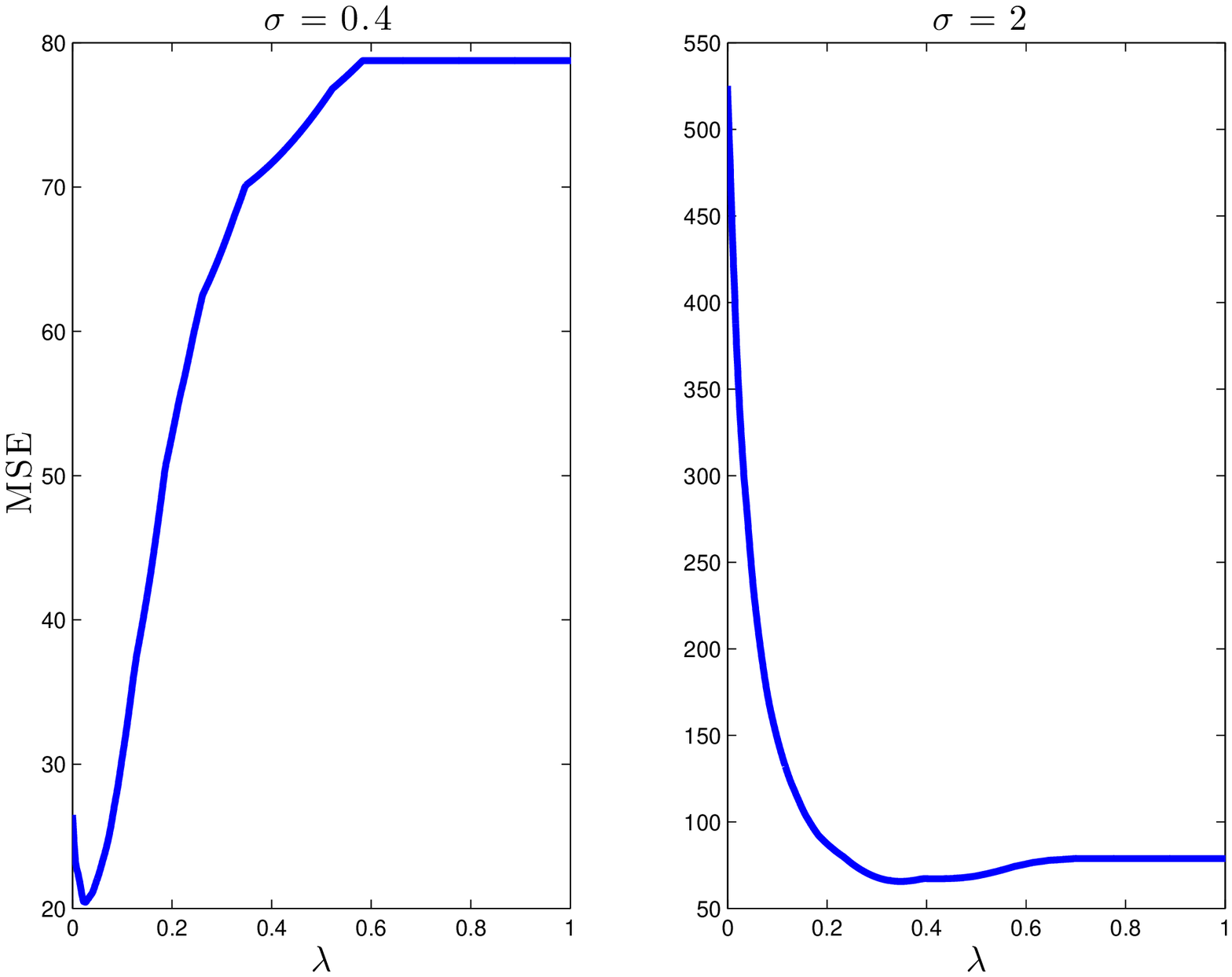}
\caption{Behavior of MSE as a function of $\lambda$ for two different noise variances.}
\label{fig:MSE}
\end{figure}

\begin{theorem}\label{thm:quasiconvex}
Let  $\{x_o(N), A(N), w(N)\}$  denotes a converging sequence of problem instances as defined in Definition \ref{def:convseq}. Suppose $A_{ij} \overset{iid}{\sim} N(0, 1/n)$. If $\hat{x}^{\lambda}(N)$ is the solution of LASSO with regularization parameter $\lambda$, then
 $$\lim_{N \rightarrow \infty} \frac{1}{N}\| \hat{x}^{\lambda}(N) - x_o\|_2^2$$ is a quasi-convex function of $\lambda$. Furthermore, if $p_X(X=0) \neq 1$, then the function is bowl-shaped.

\end{theorem}
\noindent See the proof in Section \ref{sec:proofquasiconvex}. \\
 
  In many applications such as imaging, it is important to find the optimal value of $\lambda$  that leads to the minimum MSE. We believe that the combination of the quasi-convexity of MSE in terms of $\lambda$ and certain risk estimation techniques, such as SURE, may lead to efficient algorithms for estimating $\lambda$. We leave this as an avenue for the future research. 

\subsection{Implications for AMP}\label{sec:ampimplic}

\subsubsection{AMP in asymptotic settings}
In this section we show how the result of Theorem \ref{lem:activeset} can lead to an efficient method for setting the threshold in the AMP algorithm. We first review some background on the asymptotic analysis of AMP. This section is mainly based on the results in \cite{DoMaMo09, DoMaMoNSPT, BaMo10}, and the interested reader is referred to these papers for further details. As we mentioned in Section \ref{sec:ampintro}, AMP is an iterative thresholding algorithm. Therefore, we would like to know the discrepancy of its estimate at every iteration from the original vector $x_o$. The following definition formalizes different discrepancy measures for the AMP estimates.

\begin{definition} \label{def:observables}
Let $\{x_o(N), A(N), w(N)\}$ denote a converging sequences of instances. 
Let $x^{t}(N)$ be a sequence of the estimates of AMP at iteration $t$. Consider a function $\psi: \mathbb{R}^2 \rightarrow \mathbb{R}$. An observable $J_{\psi}$ at time $t$ is defined as 
\[
J_{\psi} \left(x_o, x^{t} \right) = \lim_{N \rightarrow \infty} \frac{1}{N} \sum_{i=1}^N \psi \left(x_{o,i}(N), {x}^{t}_{i}(N) \right).
\]
\end{definition}
As before, we can consider $\psi(u,v) = (u-v)^2$ that leads to the normalized MSE of AMP at iteration $t$.  The following result that was conjectured in \cite{DoMaMo09, DoMaMoNSPT} and was finally proved in \cite{BaMo10} provides a simple description of the almost sure limits of the observables.

\begin{theorem}\label{thm:ampeqpseudo_lip} Consider the converging sequence $\{x_o(N), A(N), w(N)\}$ and let the elements of $A$ be drawn iid from $N(0,1/n)$. Suppose that ${x}^{t}(N)$ is the estimate of AMP at iteration $t$. 
Then for any pseudo-Lipschitz function $\psi: \mathbb{R}^2 \rightarrow \mathbb{R}$
\[
\lim_{N \rightarrow \infty} \frac{1}{N} \sum_i \psi \left({x}^{t}_{i}(N),{x}_{o,i} \right) = E_{X_o,W} \left[\psi(\eta(X_o+ \sigma^t W; \tau^t), X_o)\right]
\] 
almost surely, where on the right hand side $X_o$ and $W$ are two random variables with distributions $p_X$ and $N(0,1)$, respectively. $\sigma^t$ satisifies
\begin{eqnarray} \label{eq:ampevolution}
(\sigma^{t+1})^2 &=& \sigma_{\omega}^2+\frac{1}{\delta} \mathbb{E}_{X, W} \left[(\eta(X + \sigma^t W; \tau^t ) -X)^2\right], \nonumber \\
\sigma_0^2 &=& \frac{\mathbb{E} \left[X_o^2\right]}{\delta}.  
\end{eqnarray}
\end{theorem}

Similarly, our discussion of the solution of the LASSO, this theorem claims that, as long as the calculation of the pseudo-Lipschitz observables is concerned, we can assume that estimate of the AMP are modeled as iid elements with each element modeled in law as $\eta (X_{o} + \sigma^{t}W;\tau^t)$, where $X_{o} \sim p_X$ and $W \sim N(0,1)$.  As before, we are also interested in the normalized number of detections. The following theorem establishes this result.\\

\begin{theorem}\label{thm:eqampdet} Consider the converging sequence $\{x_o(N), A(N), w(N)\}$ and let the elements of $A$ be drawn iid from $N(0,1/n)$. Suppose that ${x}^{t}(N)$ is the estimate of AMP at iteration $t$. 
Then
\[  
\lim_{N \rightarrow \infty} \frac{\| {x}^{t}(N)\|_0 }{N} = \mathbb{P} (|X_o+ \sigma^t W| \geq \tau^t ) 
\]
almost surely, where on the right hand side $X_o$ and $W$ are two random variables with distributions $p_X$ and $N(0,1)$, respectively. $\sigma^t$ satisfies \eqref{eq:ampevolution}.
\end{theorem}

In other words, the result of Theorem \ref{thm:ampeqpseudo_lip} can be extended to $\psi(u,v) = I(v \neq 0)$, even though this function is not pseudo-Lipschitz. 

\subsubsection{Connection between AMP and LASSO}

The AMP algorithm in its general form can be considered as a sparse signal recovery algorithm.\footnote{To see more general form of AMP refer to  Chapter 5 of \cite{MalekiThesis}. } The choice of the threshold parameter $\tau^t$ has major impact on the performance of AMP. It turns out that if we set $\tau^t$ ``appropriately,'' then the fixed point of AMP corresponds to the solution of LASSO in the asymptotic regime. One such choice of parameters is the fixed false alarm threshold given by $\tau^t = \beta \sigma^t$, where $\sigma^t$ satisfies \eqref{eq:ampevolution}. The following result conjectured in \cite{DoMaMoNSPT, DoMaMo09} and later proved in \cite{BaMo11} formalizes this statement. 

\begin{theorem}
Consider the converging sequence $\{x_o(N), A(N), w(N)\}$ and let the elements of $A$ be drawn iid from $N(0,1/n)$. Let $x^t(N)$ be the estimate of the AMP algorithm with parameter $\tau^t =\beta \sigma^t$, where $\sigma^t$ satisfies \eqref{eq:ampevolution}. Assume that $\lim_{t \rightarrow \infty} \sigma^t =\hsig^2$. Finally, let $\hat{x}^{\lambda}$ denotes the solution of the LASSO with parameter $\lambda$ that satisfies $\lambda = \beta \hsig (1- \mathbb{P} (|X+ \hsig W| \geq \beta \hsig ))$. Then,
\begin{eqnarray*}
\lim_{t \rightarrow \infty} \lim_{N \rightarrow \infty} \frac{1}{N} \|\hat{x}^\lambda(N) - {x}^t(N)\|_2^2 =0
\end{eqnarray*}
almost surely. 
\end{theorem}

This promising result indicates that AMP can be potentially used as a fast iterative algorithm for solving the LASSO problem. However, it is not readily useful for practical scenarios in which $\sigma^t$ is not known (since neither $X$ nor its distribution are known). Therefore, in the first implementations of AMP, $\sigma^t$ has been estimated at every iteration from the observations $ x^t+ A^* z^t$. From Section \ref{sec:ampintro} we know that $v^t = x^t+ A^* z^t -x_o$ can be modeled as Gaussian $N(0, \sigma_t^2)$. Therefore, if we had access to $w^t$ we could easily estimate $\sigma^t$. However, we only observe $x^t + A^*z^t= x_0 +v^t$, and we have to estimate $\sigma^t$ from this observation. The estimates that have been proposed so far are exploiting the fact that $x_o$ is sparse and provide a biased estimate of $\sigma^t$. While such biased estimates still work well in practice, our discussion of LASSO provide an easier way to set the threshold. In the next section, based on our analysis of LASSO we discuss the performance of fixed detection thresholding policy, introduced in Section \ref{sec:ampintro}, and show that not only this thresholding policy can be implemented in its exact form, but also it has the nice properties of the fixed false alarm threshold.

\subsubsection{Fixed detection thresholding}

AMP looks for the sparsest solution of $y=Ax_o+w$ through the following iterations:
\begin{eqnarray}\label{eq:ampeq2}
x^{t+1} &=& \eta(x^t + A^* z^t; \tau^t), \nonumber \\
z^t &=& y- Ax^t + \frac{|I^t|}{n} z^{t-1}.
\end{eqnarray}
As was discussed in Section \ref{sec:intro}, a good choice for the threshold parameter $\tau^t$ is vital to the good performance of AMP. We proved in Section \ref{sec:lassopath} that the number of active elements in the solution of LASSO is a monotonic function of the parameter $\lambda$. This motivates us to set the threshold of AMP in a way that at every iteration, a certain number of coefficients remains in the active set. To understand this claim better, compare \eqref{eq:fixedpoint11} for the fixed point of LASSO and \eqref{eq:ampevolution} for the iterations of AMP. Let us replace $\hat{\tau} \triangleq\beta \hsig$ in \eqref{eq:fixedpoint11}. In addition, assume that $\lambda$ is such that $\mathbb{P} (|X+ \hsig W| \geq \beta\hsig) / \delta$ is equal to $\gamma$ for some $\gamma \in (0,1)$. Under these two assumptions, \eqref{eq:fixedpoint11} and \eqref{eq:fixedpoint21} can be converted to 
\begin{eqnarray} \label{eq:fixedpointlasso1}
\hsig^2 &=& \sigma_{\omega}^2+\frac{1}{\delta} \mathbb{E}_{X, W} \left[(\eta(X +\hsig W; \hat{\tau}) -X)^2\right], \label{eq:fixedpoint1} \nonumber \\
\lambda &=& \hat{\tau} \left(1-\gamma \right). 
\end{eqnarray}

Let us now consider the fixed point of AMP. By letting $t \rightarrow \infty$ in \eqref{eq:ampevolution} we obtain
\begin{eqnarray} \label{eq:fixedpointamp3}
\sigma_{\infty}^2 &=& \sigma_{\omega}^2+\frac{1}{\delta} \mathbb{E}_{X, W} \left[(\eta(X + \sigma_{\infty} W; \tau^\infty) -X)^2\right],
\end{eqnarray}
where $\tau^{\infty} \triangleq \lim_{t \rightarrow \infty} \tau^t$. Comparing \eqref{eq:fixedpointlasso1} and \eqref{eq:fixedpointamp3} we conclude that if we set $\tau^t$ in way that $\tau^t \rightarrow \hat{\tau}$ as $t \rightarrow \infty$ then AMP has a fixed point that corresponds to the solution of LASSO. One such approach is the fixed detection thresholding policy that was introduced in Section \ref{sec:ampintro}. According to this thresholding policy, we keep the size of the active set of AMP fixed at every iteration. Then clearly, if the algorithm converges, then final solution will have the desired number of active elements. In other words, the final solution of the AMP will also satisfy the two equations:
\begin{eqnarray} \label{eq:fixedpoint:AMPfixeddet}
\sigma_\infty^2 &=& \sigma_{\omega}^2+\frac{1}{\delta} \mathbb{E}_{X, W} (\eta(X + \sigma_\infty W; \tau^\infty) -X)^2, \label{eq:fixedpoint1} \nonumber \\
\lambda &=& \tau^\infty \left(1-\gamma \right).
\end{eqnarray}
The first question that we shall address here is wether the above two equations have a unique fixed point. Otherwise, depending on the initialization, AMP may converge to different fixed points.

 \begin{lemma}\label{lem:uniquefixedpoint}
The fixed point of \eqref{eq:fixedpoint:AMPfixeddet} is unique, i.e., for every $ 0<\gamma<1 $ there is a unique $(\sigma^\infty, \tau^\infty)$ that satisfies \eqref{eq:fixedpoint:AMPfixeddet}. 
\end{lemma}
\noindent See Section \ref{sec:proofuniquefixedpoint} for the proof of this lemma.

The heuristic discussion we have had so far shows that the fixed point of the AMP algorithm with fixed detection threshold converges to the solution of LASSO. The following theorem formalizes this result. 

\begin{theorem} \label{thm:lassoampequiv}
 Let $x^t(N)$ be an estimate of AMP with fixed detection threshold for parameter $\gamma$. Let $(\hsig, \hat{\tau})$ satisfies the fixed point equation of \eqref{eq:fixedpoint:AMPfixeddet}. In addition, let $\hat{x}^\lambda(N)$ be the solution of LASSO for $\lambda = \hat{\tau} (1 - \gamma)$. Then, we have
\[
\lim_{t \rightarrow \infty} \lim_{N \rightarrow \infty} \frac{1}{N} \|x^t- \hat{x}^\lambda \|_2 \rightarrow 0. 
\] 
\end{theorem}

As we will show in Section \ref{sec:prooflassoampequiv}, the proof of this theorem is essentially the same as the proof of Theorem 3.1 in \cite{BaMo11}. There is a slight change in the proof due to the different thresholding policy that we consider here.

\section{Proofs of the main results }\label{sec:Thms}
\subsection{Background}

\subsubsection{Quasiconvex functions}

Here, we briefly mention several properties of quasi-convex functions. For more detailed introduction, refer to section 3.4 of \cite{BoydVanderberghe}.  The following basic theorem regarding the quasi-convex functions is a key element in our proofs.

\begin{theorem}{\rm(\cite{BoydVanderberghe} Section 3.4.2)} \label{thm:quasiconvexnessuf}
A continuous function $f: \mathbb{R} \rightarrow \mathbb{R}$ is quasiconvex if and only if at least one of the following conditions holds:
\begin{itemize}
\item[1.] $f$ is non-decreasing.
\item[2.] $f$ is non-increasing.
\item[3.] There is a point $c$ in the domain of $f$ such that for $t< c$ f is non-increasing and for $t \geq c$, it is non-decreasing. 
\end{itemize}
\end{theorem} 

The following simple lemma shows that shifting and scaling preserve quasi-convexity. 

\begin{lemma} \label{lem:quascvx}
Let $a$ and $b> 0$ be two fixed numbers. Then
$f(x)$ is a quasi-convex function if and only if $g(x)=a+bf(x)$ is a quasi-convex function. 
\end{lemma}
\begin{proof}
First, assume that $f(x)$ is a quasi-convex function. According to the definition of the quasi-convexity we can write
\begin{align}
g(\alpha x+((1-\alpha)y)&=a+bf(\alpha x+((1-\alpha)y)\nonumber \\ &\leq a+b\max(f(x),f(y))\nonumber \\ &= \max(a+bf(x),a+bf(y))\nonumber \\ &=\max(g(x),g(y)).
\end{align}
Hence, $g(x)$ is quasi-convex as well. On the other hand, suppose that $g(x)$ is a quasi-convex function. Then, according to the definition, we can write
\begin{align}
f(\alpha x+((1-\alpha)y)&=\frac{g(\alpha x+((1-\alpha)y)-a}{b}\nonumber \\ &\leq \frac{\max(g(x),g(y))-a}{b} \nonumber \\ &\leq \max\left(\frac{g(x)-a}{b},\frac{g(y)-a}{b}\right) \nonumber \\ &\leq \max(f(x),f(y)).
\end{align}
Therefore, $f(x)$ is quasi-convex as well.
\end{proof}

\subsubsection{Risk of the soft thresholding function}
In this section we will review some of the basic results that have been proved elsewhere and will be used in this paper. We will also extend some of the results. As we will see these extensions will be used later in the paper. Let 
$$\Psi(\sigma^2) \triangleq \sigma_w^2 + \frac{1}{\delta} \mathbb{E}_{X,Z}\left[(\eta(X+ \sigma Z ; \beta \sigma) -X)^2\right],$$ 
where $X \sim p_X$  and $Z \sim N(0,1)$ are two independent random variables. Note that $\Psi$ is a function of $(\delta, \beta, \sigma_w^2)$, but here we assume that all these parameters are fixed, and $\Psi$ is only a function of $\sigma^2$. The following lemma is adopted from \cite{DoMaMo09}. 

\begin{lemma}
$\Psi (\sigma^2)$ is a concave function of $\sigma^2$. 
\end{lemma}

One major implication of this theorem is related to the fixed points of $\Psi(\sigma^2)$ summarized in the next lemma. This Lemma is adopted from \cite{DoMaMoNSPT}.

\begin{lemma}\label{lem:uniquefixedpointconc}
For $\sigma_w^2>0$, $\Psi$ has a unique fixed point. 
\end{lemma}

Another interesting result that will play crucial role below is the quasi-convexity of the risk of soft thresholding in terms of the threshold \cite{IainMonograph}. Let $\mu$ be a random variable with distribution $\mu \sim G$ independent of $Z \sim N(0,1)$, and define
\[
r(\tau; G) \triangleq \mathbb{E}_{\mu, Z} \left[(\eta(\mu+Z ; \tau) -\mu)^2\right].
\]
We claim that $r(\tau; G)$ is a quasi-convex function of $\tau$. It turns out that the proof of this claim is similar to the proof of quasi-convexity of soft thresholding with fixed $\mu$ \cite{IainMonograph}. However, since the version of \cite{IainMonograph} that is publicly available did not have this theorem at the time of the publication of this manuscript, and since we need to make some modifications to the proof, we provide it here. Define $\delta_0$ as a point mass at zero. For instance $G(x)= 0.5 \delta_0(x) + 0.5 \delta_0(x-1)$ is a distribution of a random variable that takes values $0$ and $1$ with probability half.

\begin{lemma}\label{lem:quasiconvexsoft}
$r( \tau; G)$ is a quasi-convex function of $\tau$.   
\end{lemma}
\begin{proof}
In order to prove this result we use Theorem \ref{thm:quasiconvexnessuf}. It is straightforward to prove that $r(\tau; G)$ is a differentiable function of $\tau$. This derivative is shown in Figure \ref{fig:GammaFunc2}. According to Theorem \ref{thm:quasiconvexnessuf}, we have to prove that the derivative either has no sign change or has one sign change from negative to positive. Note that if we had $\frac{\partial^2 r( \tau; G)}{ \partial \tau^2} \geq 0$ this would immediately show that the function is convex and hence it is also quasiconvex. However, this is not true here. Figure \ref{fig:GammaFunc2} shows $\frac{\partial}{\partial \tau} r( \tau; G)$ as a function of $\tau$. As is clear from the figure, the derivative is not strictly increasing, and hence we should not expect the second derivative to be always positive.

\begin{figure}[h!]
\includegraphics[width= 12cm]{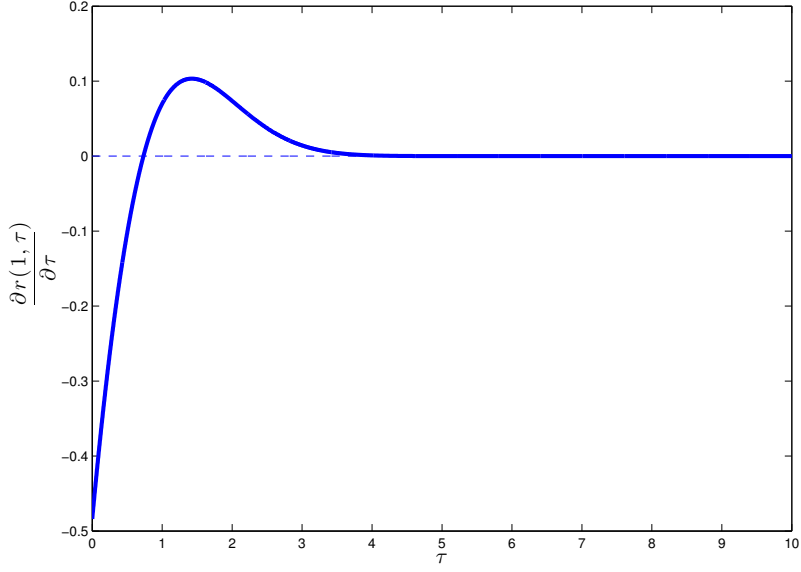}
\caption{The derivative of $\frac{\partial}{\partial \tau} r(\mu, \tau)$ as a function of $\tau$ for $\mu=1$ and $0< \tau < \infty$. Note that the derivative of the risk has only one sign change. Below that point the derivative is negative and above of that point is positive (even though it converges to zero as $\tau \rightarrow \infty$). Hence we expect the risk to be quasi-convex. }
\label{fig:GammaFunc2}
\end{figure}

 Instead we will show that  the ratio $V(\tau;G)\triangleq \frac{\frac{\partial}{\partial \tau}r(\tau; G)}{\left|\frac{\partial}{\partial\tau}r(\tau, \delta_0)\right|}$ is strictly increasing in $\tau \in [0,\infty)$. Once we prove this, we conclude that $V(\tau, \mu)$ will have at most one sign change, and since $\left|\frac{\partial}{\partial\tau}r(\tau, \delta_0)\right|$ is always positive we can conclude that $\frac{\partial r( \tau, G )}{ \partial \tau}$ will have at most one sign change. Clearly, according to Theorem \ref{thm:quasiconvexnessuf} this completes the proof of quasi-convexity. Therefore, our main goal in the rest of the proof is to show that $V(\tau,G)$ is strictly increasing. We have
\begin{align}
r(\tau, G)&=\bE_{ \mu, Z}[(\eta(\mu+Z;\tau)-\mu)^2] \nonumber \\
&=\bE_{\mu,Z}[(Z-\tau)^2\bI(\mu+Z\geq\tau)+(Z+\tau)^2\bI(\mu+Z\leq-\tau)\nonumber \\&~~~~~ \quad \quad+\mu^2\bI(|\mu+Z|<\tau)].
\end{align}
Therefore
\begin{align}\label{equ:DerRisk}
\frac{\partial r(\tau, G)}{\partial \tau}&=\bE_{\mu, Z}[-2(Z-\tau)\bI(\mu+Z\geq\tau)-(Z-\tau)^2\delta(\mu+Z-\tau)\nonumber \\ &\quad \quad+2(Z+\tau)\bI(\mu+Z\leq-\tau)-(Z+\tau)^2\delta(\tau+\mu+Z)\nonumber \\ & \quad \quad+\mu^2\delta(\tau-|\mu+Z|)] \nonumber \\ &= 
\bE_{\mu,Z}\left[2(Z+\tau)\bI(\mu+Z\leq -\tau)-2(z-\tau)\bI(\mu+Z\geq\tau)\right] \nonumber \\ 
&=\underbrace{2 \mathbb{E}_{\mu} \left[\int_{-\infty}^{-(\mu+\tau)}(z+\tau)\phi(z)dz\right]}_{\Gamma_1}-\underbrace{2 \mathbb{E}_{\mu} \left[\int_{\tau-\mu}^{\infty}(z-\tau)\phi(z)dz\right]}_{\Gamma_2}.
\end{align}
Changing integral variables to  $w=z+\tau$ in $\Gamma_1$ and $w=z-\tau$ in $\Gamma_2$ results in
\begin{align}\label{equ:DerRiskMod}
\frac{\partial r(\tau, G)}{\partial \tau}=2\mathbb{E}_{\mu} \left[\int_{-\infty}^{-\mu}w\phi(w-\tau)dw\right]+2 \mathbb{E}_{\mu}\left[ \int_{-\infty}^{\mu}w\phi(w-\tau)dw\right].
\end{align}
Consequently, we can write
\begin{align}
V(\tau;G) &= \frac{2 \mathbb{E}_{\mu} \bigg[\int_{-\infty}^{-\mu}w\phi(w-\tau)dw\bigg]+2 \mathbb{E}_{\mu} \bigg[\int_{-\infty}^{\mu}w\phi(w-\tau)dw\bigg]}{4\int_{-\infty}^{0}|w|\phi(w-\tau)dw} \nonumber\\
&=  \frac{2\mathbb{E}_{\mu}\bigg[ \int_{-\infty}^0 w \phi(w - \tau) dw -  \int_{-\mu}^0 w \phi(w- \tau) dw\bigg] }{4\int_{-\infty}^{0}|w|\phi(w-\tau)dw} \nonumber \\ 
 &~~~~+\frac{2\bE_{\mu}\bigg[ \int_{-\infty}^0 w \phi(w - \tau) dw + \int_0^\mu w \phi(w- \tau)dw\bigg]}{4\int_{-\infty}^{0}|w|\phi(w-\tau)dw}       \nonumber \\
&=\frac{ \mathbb{E}_{\mu} \bigg[\int_{-\mu}^{\mu}|w|\phi(w-\tau)dw\bigg]}{2\int_{\infty}^{0}|w|\phi(w-\tau)dw}-1.
\end{align}Let $R(\tau; G )\triangleq \frac{\mathbb{E}_{\mu}\left[ \int_{-\mu}^{\mu} |w|\phi(w-\tau)dw\right]}{\int_{-\infty}^{0}|w|\phi(w-\tau)dw}$. Taking the derivative of $R(\tau,G)$ with respect to $\tau$ gives
\begin{align}
\frac{\partial R(\tau,G)}{\partial \tau}&=\frac{\mathbb{E}_{\mu}\left[\int_{-\mu}^{\mu}(w-\tau)|w|\phi(w-\tau)dw\right]\int_{-\infty}^0 |w|\phi(w-\tau)dw}{\left(\int_{-\infty}^{0}|w|\phi(w-\tau)dw\right)^2} \nonumber \\&~~~~-\frac{\mathbb{E}_{\mu}\left[\int_{-\mu}^{\mu}|w|\phi(w-\tau)dw\right] \int_{-\infty}^{0}(w-\tau)|w|\phi(w-\tau)dw} {\left(\int_{-\infty}^{0}|w|\phi(w-\tau)dw\right)^2} \nonumber \\ &=\frac{\overbrace{\mathbb{E}_{\mu}\left[ \int_{-\mu}^{\mu}w|w|\phi(w-\tau)dw\right]\ }^{\Gamma_1}\int_{-\infty}^0 |w|\phi(w-\tau)dw}{\left(\int_{-\infty}^{0}|w|\phi(w-\tau)dw\right)^2} \nonumber \\&~~~~ +\frac{\overbrace{ \mathbb{E}_{\mu}\left[ \int_{-\mu}^{\mu}|w|\phi(w-\tau)dw\right]\int_{0}^{\infty}w|w|\phi(w+\tau)dw}^{\Gamma_2}}{\left(\int_{-\infty}^{0}|w|\phi(w-\tau)dw\right)^2}. 
\end{align}
 $\Gamma_1$ can be simplified to
\begin{align}
\Gamma_1= \mathbb{E}_{\mu} \left[\int_0^{|\mu|}w^2(\phi(w-\tau)-\phi(w+\tau))dw\right]>0.
\end{align}
Therefore, since $\Gamma_2>0$, we have $\frac{\partial R(\tau,G)}{\partial \tau}>0$ that proves the fact that $R(\tau; G)$ is an increasing function. This in addition with the fact that $V(0; G) <0$ completes the proof of quasi-convexity. 
\end{proof}

So far, we have been able to prove that $\frac{\partial V(\tau, G)}{\partial \tau} \geq 0$, which proves (when combined by the fact that $V(0, G) <0$) the quasi-convexity of the risk. However, this does not mean that the function is bowl-shaped yet (See Definition \ref{def:Qcvx}). Infact, the function is bowl-shaped if sign-change happens at certain values of $\tau$. We would now like to prove that the zero crossing in fact happens if $G \neq \delta_0$. 
\begin{lemma} \label{cor:ZC}
$\frac{\partial r(\tau, G)}{\partial \tau}\big|_{\tau=0}<0$
\end{lemma}
\begin{proof}
We have
\begin{align}
\frac{\partial r(\tau, G)}{\partial \tau}\bigg|_{\tau=0}&=2\mathbb{E}_{\mu}\left[\int_{-\infty}^{-\mu}z\phi(z)dz-\int_{-\mu}^{\infty}z\phi(z)dz\right] \nonumber \\
&=2\mathbb{E}_{\mu}\left[-\frac{2e^{\frac{-\mu^2}{2}}}{\sqrt{2\pi}}\right] <0,
\end{align}
\end{proof}
If we prove that for large enough $\hat{\tau}$ we have $\frac{\partial r(\tau, G)}{\partial \tau}\big|_{\tau=\hat{\tau}}>0$, then we can conclude that $\frac{\partial r(\tau, G)}{\partial \tau}$ has at least one zero-crossing. On the other hand, according to Lemma \ref{lem:quasiconvexsoft}, $\frac{\partial r(\tau, G)}{\partial \tau}$ has at most one zero crossing. Therefore, we would actually prove that $\frac{\partial r(\tau, G)}{\partial \tau}$ has exactly one zero-crossing and, consequently, the risk of soft-thresholding function is a bowl-shaped function.
We require the following two lemmas in the proof of our main claim that is summarized in Proposition \ref{pro:ZC}.

\begin{lemma}\label{lem:eps1}
There exists a $\epsilon_0>0$ such that $\mathbb{P}(\mu>\epsilon_0)>0$.
\end{lemma}
\noindent This lemma is a straightforward consequence of the assumption $\mathbb{P} (\mu \neq 0) \neq0$.  \\

\begin{lemma}\label{lem:eps2}
There exists $ \mu_0>\epsilon_0$ given in Lemma \ref{lem:eps1} such that $\mathbb{P}(\mu>\mu_0)>0$.
\end{lemma}
\begin{proof}
We again prove this lemma using contradiction. Suppose that such $\mu_0$ does not exist. Therefore, for every $\mu_0>\epsilon_0$ we have $\mathbb{P}(\mu>\mu_0)=0$. This would give $\mathbb{P}(\mu>\epsilon_0)=0$. However, this is in contradiction with the property of $\epsilon_0$ in Lemma \ref{lem:eps1} and hence the proof is complete.
\end{proof}

\begin{proposition}\label{pro:ZC}
If $\mathbb{P} (\mu \neq 0) \neq 0$. Let $\epsilon_0$ and $\mu_0$ be defined as in Lemmas \ref{lem:eps1} and \ref{lem:eps2}. Then $\frac{\partial r(\tau, G)}{\partial \tau}\big|_{\tau=\hat{\tau}}>0$ for $\hat{\tau}>\frac{1}{\epsilon_0}\ln\left(\frac{c_1}{c_2}\right)$, where $c_2$ is defined in \eqref{eq:c_2}.
\end{proposition}
\begin{proof} Without loss of generality and for the simplicity of notation, we assume that $\mathbb{P} (\mu \geq 0) =1$.
Using the same change of variable technique as in (\ref{equ:DerRiskMod}), we have
\begin{align}
\frac{\partial r(\tau, G)}{\partial \tau}&=2\mathbb{E}_{\mu} \left[\int_{-\infty}^{-\mu}w\phi(w-\tau)dw\right]+2 \mathbb{E}_{\mu}\left[ \int_{-\infty}^{\mu}w\phi(w-\tau)dw\right] \nonumber \\
&=2\mathbb{E}_{\mu} \left[\int_{-\infty}^{-\mu}w\phi(w-\tau)dw+\int_{-\infty}^{0}w\phi(w-\tau)dw \right] \nonumber \\ 
&~~~~+2 \mathbb{E}_{\mu}\left[ \int_{0}^{\mu}w\phi(w-\tau)dw\right] \nonumber \\
&\geq 4\underbrace{\mathbb{E}_{\mu} \left[\int_{-\infty}^{0}w\phi(w-\tau)dw\right]}_{\Gamma_1}+2\underbrace{\mathbb{E}_{\mu}\left[ \int_{0}^{\mu}w\phi(w-\tau)dw\right]}_{\Gamma_2}.
\end{align}
Note that $\Gamma_1 <0$, and $\Gamma_2 >0$. Our goal is to show that for large values of $\tau$, $\Gamma_2 > |\Gamma_1|$. To achieve this goal,
we find an upper bound for $|\Gamma_1|$ and a lower bound for $\Gamma_2$. Simplifying $\Gamma_1$ gives
\begin{align}\label{eq:Gamma1Simp}
\Gamma_1&=\mathbb{E}_{\mu} \left[\int_{-\infty}^{0}w\phi(w-\tau)dw\right] \nonumber \\
&=\frac{2e^{\frac{-\tau^2}{2}}}{\sqrt{2\pi}}\int_{-\infty}^{0}we^{\frac{-w^2}{2}+w\tau}dw \nonumber \\
&\geq \frac{2e^{\frac{-\tau^2}{2}}}{\sqrt{2\pi}}\int_{-\infty}^{0}we^{\frac{-w^2}{2}}dw \nonumber \\
& = -\frac{e^{\frac{-\tau^2}{2}}}{\sqrt{2\pi}}\int_{-\infty}^{0}|w|e^{\frac{-w^2}{2}}dw \nonumber \\
&\geq -\sqrt{\frac{2}{\pi}} e^{\frac{-\tau^2}{2}}.
\end{align}
Therefore, it is sufficient to prove that $\Gamma_2> \sqrt{\frac{2}{\pi}}  e^{\frac{-\tau^2}{2}}$. To achieve this goal, we first use the assumptions that $G\neq \delta_0$ and $\mu\geq 0$ to prove the following two lemmas.

Using $\epsilon_0$ and $\mu_0$ introduced in Lemmas \ref{lem:eps1} and \ref{lem:eps2}, we can obtain a lower bound for $\Gamma_2$ as the following:
\begin{align}\label{eq:Gamma2Simp}
\Gamma_2&=\mathbb{E}_{\mu} \left[\int_{0}^{\mu}w\phi(w-\tau)dw\right] \nonumber \\
&=\int_0^\infty \int_0^\mu w\phi(w-\tau)dwdG(\mu) \nonumber \\
&\geq \int_{\epsilon_0}^\infty \int_{\epsilon_0}^\mu w\phi(w-\tau)dwdG(\mu) \nonumber \\
&=\frac{e^{\frac{-\tau^2}{2}}}{\sqrt{2\pi}}\int_{\epsilon_0}^\infty \int_{\epsilon_0}^\mu we^{\frac{-w^2}{2}+w\tau}dwdG(\mu) \nonumber \\
&\geq \frac{e^{\frac{-\tau^2}{2}}e^{\epsilon_0 \tau}}{\sqrt{2\pi}}\int_{\epsilon_0}^\infty \int_{\epsilon_0}^\mu we^{\frac{-w^2}{2}}dwdG(\mu) \nonumber \\
&=\frac{e^{\frac{-\tau^2}{2}}e^{\epsilon_0 \tau}}{\sqrt{2\pi}}\int_{\epsilon_0}^\infty \left(e^{\frac{-\epsilon_0^2}{2}}-e^{\frac{-\mu^2}{2}}\right)dG(\mu) \nonumber \\
&\geq \frac{e^{\frac{-\tau^2}{2}}e^{\epsilon_0 \tau}}{\sqrt{2\pi}}\int_{\mu_0}^\infty \left(e^{\frac{-\epsilon_0^2}{2}}-e^{\frac{-\mu^2}{2}}\right)dG(\mu) \nonumber \\
&\geq \frac{e^{\frac{-\tau^2}{2}}e^{\epsilon_0 \tau}}{\sqrt{2\pi}}\int_{\mu_0}^\infty \left(e^{\frac{-\epsilon_0^2}{2}}-e^{\frac{-\mu_0^2}{2}}\right)dG(\mu) \nonumber \\
&=e^{\frac{-\tau^2}{2}}e^{\epsilon_0 \tau}c_2,
\end{align}
where
\begin{equation}\label{eq:c_2}
c_2 \triangleq  \left(e^{\frac{-\epsilon_0^2}{2}}-e^{\frac{-\mu_0^2}{2}}\right) \mathbb{P}(\mu > \mu_0).
\end{equation}
Combining (\ref{eq:Gamma1Simp}) and (\ref{eq:Gamma2Simp}) would give 
\begin{align}
\Gamma_1+\Gamma_2 \geq \underbrace{(c_2 e^{\epsilon_0 \tau}-c_1) e^{\frac{-\tau^2}{2}}.}_{\Lambda}
\end{align} 
Therefore, $\Lambda$ would be positive in a case that $\tau$ is large enough, i.e., if $\tau>\frac{1}{\epsilon_0}\ln\left(\frac{c_1}{c_2}\right)$. Therefore, the proof is complete.
\end{proof}
Combining Corollary \ref{cor:ZC}, Proposition \ref{pro:ZC}, and the fact that $\frac{\partial r(\tau,G)}{\partial \tau}$ has at most one zero-crossing proves that $\frac{\partial r(\tau,G)}{\partial \tau}$ has exactly one zero-crossing from negative to positive, and hence $r(G,\tau)$ is a quasi-convex and bowl-shaped function.

\begin{proposition}
$r(G,\tau)$ is a strictly decreasing function of $\tau$ for the case $G=\delta_0$.
\end{proposition}
\begin{proof}
We can write
\begin{align}
r(\delta_0,\tau)&=\bE\left[\eta^2(Z;\tau)\right] \nonumber \\
&=\bE\left[(Z-\tau)^2\bI(Z>\tau)+(Z+\tau)^2\bI(Z<-\tau)\right] \nonumber \\
&=\int_{\tau}^{\infty}(z-\tau)^2\phi(z)dz+\int_{-\infty}^{-\tau}(z+\tau)^2\phi(z)dz
\end{align}
Therefore, taking derivative with respect to $\tau$ gives
\begin{align}
\frac{dr(\delta_0,\tau)}{d\tau}=-4\int_{\tau}^{\infty}z\phi(z)dz<0. \nonumber 
\end{align}
\end{proof}

\subsection{Proof of Theorem \ref{lem:activeset}}\label{sec:proofactiveset}
First note that according to Theorem \ref{thm:lassodetstate} we have
\[
\lim_{N \rightarrow \infty} \frac{1}{N} \sum_{i=1}^N \mathbb{I} (\hat{x}^{\lambda}_{ i}(N) \neq 0) = \mathbb{P} (|\eta(X_o + \hat{\sigma} Z ; \beta \hat{\sigma}) | > 0),
\]
where $(\hat{\sigma},\lambda,\beta)$ satisfy the following equations: 
\begin{align}\label{equ:FixedPoint1}
&\hat{\sigma}^2=\sigma_z^2+\frac{1}{\delta}\bE_{X,Z}\left[(\eta(X+\hat{\sigma}Z;\hat{\sigma}\beta)-X)^2\right], \nonumber \\
& \lambda=\beta\hat{\sigma}\left(1-\frac{1}{\delta}\bP(|X+\hat{\sigma}Z|>\hat{\sigma}\beta)\right).
\end{align}
Therefore, we have to prove that $\frac{ d  \mathbb{P}(|\eta(X_o + \hat{\sigma} Z ; \beta \hat{\sigma}) | > 0) }{d \lambda} <0$. The main difficulty of this problem is clear from \eqref{equ:FixedPoint1}: $(\beta, \hat{\sigma}, \lambda)$ are complicated functions of each other whose explicit formulations are not known. Using the chain rule we have
\begin{eqnarray}\label{equ:ChainRule}
\frac{d}{d\lambda}\bP(|\eta(X+\hat{\sigma}Z;\beta\hat{\sigma})|>0)=\frac{d}{d\beta}\bP(|\eta(X+\hat{\sigma}Z;\beta\hat{\sigma})|>0)\frac{d\beta}{d\lambda}.
\end{eqnarray}
This simple expression enables us to break the proof into the following two simpler parts:
\begin{enumerate}
\item We prove that $\Big( \frac{d}{d\beta}\bP(|\eta(X+\hat{\sigma}Z;\beta\hat{\sigma})|>0) \Big) <0$. 
\item We prove that $\frac{d \beta} {d \lambda} >0$.
\end{enumerate}
Combining these two results with \eqref{equ:ChainRule} completes the proof. 

\begin{lemma}\label{lem:detovertau}
Let $(\beta, \hat{\sigma})$ satisfy \eqref{equ:FixedPoint1}. Then 
$$\Big( \frac{d}{d\beta}\bP(|\eta(X+\hat{\sigma}Z;\beta\hat{\sigma})|>0) \Big) <0.$$
\end{lemma}

\begin{proof}
Since $|\eta(X+\hat{\sigma}Z;\beta\hat{\sigma})|>0$ if and only if $|\frac{X}{\hat{\sigma}}+Z|>\beta$, it is sufficient to prove $\frac{d}{d\beta}\bP(|\frac{X}{\hat{\sigma}}+Z|>\beta)<0$. We have
\begin{align}\label{equ:Derivative}
\frac{d}{d\beta}\bP\left(\left|\frac{X}{\hat{\sigma}}+Z\right|>\beta\right)=\frac{\partial}{\partial\hat{\sigma}}\bP\left(\left|\frac{X}{\hat{\sigma}}+Z\right|>\beta\right)\frac{d\hat{\sigma}}{d\beta}+\frac{\partial}{\partial\beta}\bP\left(\left|\frac{X}{\hat{\sigma}}+Z\right|>\beta \right).
\end{align}
The rest of the proof has four main steps:
\begin{enumerate}
\item Calculation of $\frac{\partial}{\partial\hat{\sigma}}\bP\left(\left|\frac{X}{\hat{\sigma}}+Z\right|>\beta\right)$.
\item Calculation of $\frac{\partial}{\partial\beta}\bP\left(\left|\frac{X}{\hat{\sigma}}+Z\right|>\beta \right)$.
\item Calculation of $\frac{d\hat{\sigma}}{d\beta}$. 
\item Plugging the results of the above three steps in \eqref{equ:Derivative} and proving that $\Big( \frac{d}{d\beta}\bP\left(\left|\frac{X}{\hat{\sigma}}+Z\right|>\beta\right) \Big) <0$. 
\end{enumerate}
Here are these four steps realized: \\

\textbf{Step 1: Calculation of $\frac{\partial}{\partial\hat{\sigma}}\bP\left(\left|\frac{X}{\hat{\sigma}}+Z\right|>\beta\right)$} \\
\vspace{0.2 cm}

Simple algebra leads us to 
\begin{align}
\frac{\partial}{\partial \hat{\sigma}} \mathbb{P} \left( \left| \frac{X}{\hsig} +Z \right| \geq \beta \right)&= \mathbb{E}_X\left[ \frac{\partial}{\partial \hsig}\bE_Z\left[\bI\left(\left|\frac{X}{\hsig}+Z\right|\geq\beta\right)\bigg|X\right]\right] \nonumber \\ & =\bE_X\left[\frac{\partial}{\partial \hsig}\left(\int_{\beta-\frac{X}{\hsig}}^{\infty}\phi(z)dz+\int_{-\infty}^{-\beta-\frac{X}{\hsig}}\phi(z)dz\right)\right] \nonumber \\ &=\bE_X\left[\frac{X}{\hsig^2}\phi\left(\frac{X}{\hsig}+\beta\right)-\frac{X}{\hsig^2}\phi\left(\frac{X}{\hsig}-\beta\right)\right].
\end{align}

\textbf{Step 2: Calculation of $\frac{\partial}{\partial\beta}\bP\left(\left|\frac{X}{\hat{\sigma}}+Z\right|>\beta \right)$} \\
\vspace{0.2 cm}

Similar to Step 1, we can write
\begin{align}
\frac{\partial}{\partial \beta} \mathbb{P} \left( \left| \frac{X}{\hsig} +Z \right| \geq \beta \right)&= \mathbb{E}_X\left[ \frac{\partial}{\partial \beta}\bE_Z\left[\bI\left(\left|\frac{X}{\hsig}+Z\right|\geq\beta\right)\bigg|X\right]\right] \nonumber \\ & =\bE_X\left[\frac{\partial}{\partial \beta}\left(\int_{\beta-\frac{X}{\hsig}}^{\infty}\phi(z)dz+\int_{-\infty}^{-\beta-\frac{X}{\hsig}}\phi(z)dz\right)\right] \nonumber \\ &=\bE_X\left[-\phi\left(\frac{X}{\hsig}+\beta\right)-\phi\left(\frac{X}{\hsig}-\beta\right)\right].
\end{align}

\textbf{Step 3: Calculation of $\frac{d\hat{\sigma}}{d\beta}$} \\
\vspace{0.2 cm}

We can rewrite (\ref{equ:FixedPoint1}) as 
\begin{eqnarray}\label{equ:ReFixedPoint}
\delta=\frac{\delta \sigma_z^2}{\hsig^2}+\bE_{X,Z}\left[\left(\eta\left(\frac{X}{\hsig}+Z;\beta\right)-\frac{X}{\hsig}\right)^2\right].
\end{eqnarray}
Taking derivative with respect to $\beta$ from both sides of (\ref{equ:ReFixedPoint}) yields
\begin{align}
0&=\frac{-2\delta\sigma_z^2}{\hsig^3}\frac{d\hsig}{d\beta}+\frac{d}{d \beta}\bE_{X,Z}\left[\left(\eta\left(\frac{X}{\hsig}+Z;\beta\right)-\frac{X}{\hsig}\right)^2\right] \nonumber \\
&=\frac{-2\delta\sigma_z^2}{\hsig^3}\frac{d\hsig}{d\beta}+\frac{\partial}{\partial \beta}\bE_{X,Z}\left[\left(\eta\left(\frac{X}{\hsig}+Z;\beta\right)-\frac{X}{\hsig}\right)^2\right] \nonumber \\ &~~~~+\frac{\partial}{\partial \hsig}\bE_{X,Z}\left[\left(\eta\left(\frac{X}{\hsig}+Z;\beta\right)-\frac{X}{\hsig}\right)^2\right] \frac{d\hsig}{d\beta} \nonumber \\
&=\frac{-2\delta\sigma_z^2}{\hsig^3}\frac{d\hsig}{d\beta}+2\bE_{X,Z}\left[\left(\eta\left(\frac{X}{\hsig}+Z;\beta\right)-\frac{X}{\hsig}\right)\frac{\partial}{\partial \beta}\left(\eta\left(\frac{X}{\hsig}+Z;\beta\right)-\frac{X}{\hsig}\right)\right] \nonumber \\ &~~~~+2\bE_{X,Z}\left[\left(\eta\left(\frac{X}{\hsig}+Z;\beta\right)-\frac{X}{\hsig}\right)\frac{\partial}{\partial \hsig}\left(\eta\left(\frac{X}{\hsig}+Z;\beta\right)-\frac{X}{\hsig}\right)\right] \frac{d\hsig}{d\beta} \nonumber \\
&=\frac{-2\delta\sigma_z^2}{\hsig^3}\frac{d\hsig}{d\beta}\nonumber \\
&~~~~+2\bE_{X,Z}\bigg[\left(\eta\left(\frac{X}{\hsig}+Z;\beta\right)-\frac{X}{\hsig}\right)\bigg(-\bI\left(\frac{X}{\hsig}+Z>\beta\right)\nonumber \\
&\qquad \qquad \quad~~-\left(\frac{X}{\hsig}+Z-\beta\right)\delta\left(\frac{X}{\hsig}+Z-\beta\right)\nonumber \\ &
\qquad \qquad \quad~~+\bI\left(\frac{X}{\hsig}+Z<-\beta\right)-\left(\frac{X}{\hsig}+Z+\beta\right)\delta\left(\frac{X}{\hsig}+Z+\beta\right)\bigg)\bigg]\nonumber \\
&~~~~+2\bE_{X,Z}\bigg[\left(\eta\left(\frac{X}{\hsig}+Z;\beta\right)-\frac{X}{\hsig}\right)\bigg(\left(\frac{-X}{\hsig^2}\right)\bI\left(\frac{X}{\hsig}+Z>\beta\right)\nonumber \\
&\qquad \qquad \quad~~+\left(\frac{X}{\hsig}+Z-\beta\right)\delta\left(\frac{X}{\hsig}+Z-\beta\right)\left(\frac{-X}{\sigma^2}\right)\nonumber \\
&\qquad \qquad \quad~~+\left(\frac{-X}{\hsig^2}\right)\bI\left(\frac{X}{\hsig}+Z<-\beta\right)\nonumber \\
&\qquad \qquad \quad~~+\left(\frac{X}{\hsig}+Z+\beta\right)\delta\left(\frac{X}{\hsig}+Z+\beta\right)\left(\frac{X}{\hsig^2}\right)\bigg)\bigg] \frac{d\hsig}{d\beta}\nonumber \\ 
&=\frac{-\delta\sigma_z^2}{\hsig^3}\frac{d\hsig}{d\beta}+\bE_{X,Z}\left[\frac{X}{\hsig^2}\left(\eta\left(\frac{X}{\hsig}+Z;\beta\right)-\frac{X}{\hsig}\right)\left(\bI\left(\Big|\frac{X}{\hsig}+Z\Big|<\beta\right)\right)\right]\frac{d\hsig}{d\beta}
\nonumber \\ & ~~~~+\bE_{X,Z}\bigg[\left(\eta\left(\frac{X}{\hsig}+Z;\beta\right)-\frac{X}{\hsig}\right)\nonumber \\ &\qquad \qquad \quad~~\left(-\bI\left(\frac{X}{\hsig}+Z>\beta\right)+\bI\left(\frac{X}{\hsig}+Z<-\beta\right)\right)\bigg].\end{align}
Therefore, we can write
\begin{align}\label{equ:FinalDer}
\frac{d\hsig}{d\beta}&=\frac{E_{X,Z}\left[\left(\eta\left(\frac{X}{\hsig}+Z;\beta\right)-\frac{X}{\hsig}\right)\left(-\bI\left(\frac{X}{\hsig}+Z>\beta\right)+\bI\left(\frac{X}{\hsig}+Z<-\beta\right)\right)\right]}{\frac{\delta\sigma_z^2}{\hsig^3}+\bE_{X,Z}\left[\frac{X^2}{\hsig^3}\left(\bI\left(\Big|\frac{X}{\hsig}+Z\Big|<\beta\right)\right)\right]} \nonumber \\ 
&=\frac{\bE_X\left[\int_{\beta-\frac{X}{\hsig}}^{\infty}(\beta-z)\phi(z)dz+\int_{-\infty}^{-\beta-\frac{X}{\hsig}}(z+\beta)\phi(z)dz\right]}{\frac{\delta\sigma_z^2}{\hsig^3}+\bE_{X,Z}\left[\frac{X^2}{\hsig^3}\left(\bI\left(\Big|\frac{X}{\hsig}+Z\Big|<\beta\right)\right)\right]}.
\end{align}

\noindent \textbf{Step 4:  Proving that $\Big( \frac{d}{d\beta}\bP\left(\left|\frac{X}{\hat{\sigma}}+Z\right|>\beta\right) \Big) <0$} \\
\vspace{0.2 cm}

\noindent If we apply (\ref{equ:FinalDer}) to (\ref{equ:Derivative}) we obtain
\begin{align}\label{equ:DerivativeContinued}
&\frac{d}{d\beta}\bP\left(\left|\frac{X}{\hat{\sigma}}+Z\right|>\beta\right) \nonumber \\
&=\frac{d}{d\beta}\bE_{X,Z}\left[\bI\left(\left|\frac{X}{\hsig}+Z\right|>\beta\right)\right] \nonumber \\
&=\frac{\partial}{\partial \beta}\bE_{X,Z}\left[\int_{\beta-\frac{X}{\hsig}}^{\infty} \phi(z)dz+\int_{-\infty}^{-\beta-\frac{X}{\hsig}}\phi(z)dz\right]\nonumber \\ &~~~~+\frac{\partial}{\partial \hsig}\bE_{X,Z}\left[\int_{\beta-\frac{X}{\hsig}}^{\infty} \phi(z)dz +\int_{-\infty}^{-\beta-\frac{X}{\hsig}}\phi(z)dz\right]\frac{d\hsig}{d\beta}  \nonumber \\ &=-\overbrace{\frac{\bE_{X,Z}\left[\phi\left(\beta-\frac{X}{\hsig}\right)+\phi\left(\beta+\frac{X}{\hsig}\right)\right]\frac{\delta\sigma_z^2}{\hsig^3}}{\frac{\delta\sigma_z^2}{\hsig^3}+\bE_{X,Z}\left[\frac{X^2}{\hsig^3}\left(\bI\left(\Big|\frac{X}{\hsig}+Z\Big|<\beta\right)\right)\right]}}^{\Lambda_1}\nonumber \\&~~~~-\overbrace{\frac{\bE_{X,Z}\left[\phi\left(\beta-\frac{X}{\hsig}\right)+\phi\left(\beta+\frac{X}{\hsig}\right)\right] \bE_{X,Z}\left[\frac{X^2}{\hsig^3}\left(\bI\left(\Big|\frac{X}{\hsig}+Z\Big|<\beta\right)\right)\right]}{\frac{\delta\sigma_z^2}{\hsig^3}+\bE_{X,Z}\left[\frac{X^2}{\hsig^3}\left(\bI\left(\Big|\frac{X}{\hsig}+Z\Big|<\beta\right)\right)\right]}}^{\Lambda_2} \nonumber \\&~~~~+\overbrace{\frac{\bE_{X,Z}\left[\frac{X}{\hsig^2}\left(-\phi\left(\beta-\frac{X}{\hsig}\right)+\phi\left(\beta+\frac{X}{\hsig}\right)\right)\right]\beta\left(\int_{\beta-\frac{X}{\hsig}}^{\infty}\phi(z)dz+\int_{-\infty}^{-\beta-\frac{X}{\hsig}}\phi(z)dz\right)}{\frac{\delta\sigma_z^2}{\hsig^3}+\bE_{X,Z}\left[\frac{X^2}{\hsig^3}\left(\bI\left(\Big|\frac{X}{\hsig}+Z\Big|<\beta\right)\right)\right]}}^{\Lambda_3}\nonumber \\&~~~~+\overbrace{\frac{\bE_{X,Z}\left[\frac{X}{\hsig^2}\left(-\phi\left(\beta-\frac{X}{\hsig}\right)+\phi\left(\beta+\frac{X}{\hsig}\right)\right)\right]\left(-\int_{\beta-\frac{X}{\hsig}}^{\infty}z\phi(z)dz+\int_{-\infty}^{-\beta-\frac{X}{\hsig}}z\phi(z)dz\right)}{\frac{\delta\sigma_z^2}{\hsig^3}+\bE_{X,Z}\left[\frac{X^2}{\hsig^3}\left(\bI\left(\Big|\frac{X}{\hsig}+Z\Big|<\beta\right)\right)\right]}}^{\Lambda_4}.
\end{align}
Considering terms in (\ref{equ:DerivativeContinued}), we can write
\begin{itemize}
\item[(i)] $\Lambda_1 \geq 0$ since all the terms in $\Lambda_1$ are positive.
\item[(ii)] $\Lambda_3 \leq 0$ since for $X<0$ we have  $\phi\left(\frac{X}{\hsig}+\beta \right)>\phi\left(\frac{X}{\hsig}-\beta\right)$ and for $X>0$ we have $\phi\left(\frac{X}{\hsig}+\beta \right)<\phi\left(\frac{X}{\hsig}-\beta\right)$.
\item[(iii)] For the other two terms, $\Lambda_2$ and $\Lambda_4$, we have 
\begin{align}\label{equ:Simplify2}
& \Lambda_2+\Lambda_4 \nonumber \\
&= -\frac{\bE_{X}\left[\left(\phi(\beta-\frac{X}{\hsig})+\phi(\beta+\frac{X}{\hsig})\right)\left(\frac{X}{\hsig^2}\int_{-\beta-\frac{X}{\hsig}}^{\beta-\frac{X}{\hsig}}z\phi(z)dz+\frac{X^2}{\hsig^3}\int_{-\beta-\frac{X}{\hsig}}^{\beta-\frac{X}{\hsig}}\phi(z)dz\right)\right]}{\frac{\delta\sigma_z^2}{\hsig^3}+\bE_{X,Z}\left[\frac{X^2}{\hsig^3}\left(\bI\left(\Big|\frac{X}{\hsig}+Z\Big|<\beta\right)\right)\right]}  
\nonumber \\ &=-\frac{\bE_{X}\left[\left(\phi(\beta-\frac{X}{\hsig})+\phi(\beta+\frac{X}{\hsig})\right)\left(\frac{X}{\hsig^2}\int_{-\beta-\frac{X}{\hsig}}^{\beta-\frac{X}{\hsig}}\left(z+\frac{X}{\hsig}\right)\phi(z)dz\right)\right]}{\frac{\delta\sigma_z^2}{\hsig^3}+\bE_{X,Z}\left[\frac{X^2}{\hsig^3}\left(\bI\left(\Big|\frac{X}{\hsig}+Z\Big|<\beta\right)\right)\right]}
\nonumber \\  &=-\frac{\bE_{X}\left[\left(\phi(\beta-\frac{X}{\hsig})+\phi(\beta+\frac{X}{\hsig})\right)\overbrace{\left(\frac{X}{\hsig^2}\int_{-\beta}^{\beta}w\phi\left(w-\frac{X}{\hsig}\right)dz\right)}^{\Upsilon}\right]}{\frac{\delta\sigma_z^2}{\hsig^3}+\bE_{X,Z}\left[\frac{X^2}{\hsig^3}\left(\bI\left(\Big|\frac{X}{\hsig}+Z\Big|<\beta\right)\right)\right]}.
\end{align}
Regarding the $\Upsilon$ in (\ref{equ:Simplify2}), for any $\epsilon \in [-\beta,0)$:
\begin{itemize}
\item[-] If $X<0$ then $\epsilon\left(\phi\left(\epsilon-\frac{X}{\hsig}\right)-\phi\left(-\epsilon-\frac{X}{\hsig}\right)\right)<0$.
\item[-] If $X>0$  then $\epsilon\left(\phi\left(\epsilon-\frac{X}{\hsig}\right)-\phi\left(-\epsilon-\frac{X}{\hsig}\right)\right)>0$.
\end{itemize}
 Therefore, $\Upsilon>0$ and consequently $\Lambda_2+\Lambda_4<0$. 
\end{itemize}
Putting (i), (ii), and (iii) together we conclude that
\begin{eqnarray}\label{equ:ChainRule1}
\frac{d}{d\beta}\bP\left(\left|\frac{X}{\hsig}+Z\right| >0 \right)<0.
\end{eqnarray}
\end{proof}

\begin{lemma}\label{lem:lambdataumon}
Let $(\beta, \hat{\sigma}, \lambda)$ satisfy \eqref{equ:FixedPoint1}. Then $\frac{d \lambda}{d \beta} >0$. 
\end{lemma}
\begin{proof}
 Taking derivative with respect to $\beta$ from both sides of (\ref{equ:FixedPoint1}) yields
\begin{align}
\frac{d\lambda}{d\beta}=\underbrace{\frac{d(\beta \hsig)}{d\beta}\left(1-\frac{1}{\delta}\bP(|X+\hat{\sigma}Z|>\hat{\sigma}\beta)\right)}_{\Theta_1}-\underbrace{\beta\hsig\frac{d}{d\beta}\bP(|X+\hat{\sigma}Z|>\hat{\sigma}\beta)}_{\Theta_2}.
\end{align}
According to Lemma \ref{lem:detovertau}, $\Theta_2<0$. Regarding $\Theta_1$, since $\lambda \geq 0$ satisfies (\ref{equ:FixedPoint1}), $\left(1-\frac{1}{\delta}\bP(|X+\hat{\sigma}Z|>\hat{\sigma}\beta)\right)>0$, and as a result it is sufficient to prove that $\frac{d(\beta\hsig)}{d\beta}\geq 0 $. We have 
\begin{align}\label{equ:DerTauSigma}
&\frac{d(\beta\hsig)}{d\beta}=\hsig+\beta \frac{d\hsig}{d\beta}\nonumber \\ &\overset{(a)}{=} \hsig+\beta\frac{E_{X,Z}\left[\left(\eta\left(\frac{X}{\hsig}+Z;\beta\right)-\frac{X}{\hsig}\right)\left(-\bI\left(\frac{X}{\hsig}+Z>\beta\right)+\bI\left(\frac{X}{\hsig}+Z<-\beta\right)\right)\right]}{\frac{\delta\sigma_z^2}{\hsig^3}+\bE_{X,Z}\left[\frac{X^2}{\hsig^3}\left(\bI\left(\Big|\frac{X}{\hsig}+Z\Big|<\beta\right)\right)\right]} \nonumber \\ &=\frac{\frac{\delta\sigma_z^2}{\hsig^2}+\bE_{X,Z}\left[\frac{X^2}{\hsig^2}\left(\bI\left|\frac{X}{\hsig}+Z\right|<\beta\right)\right]}{\frac{\delta\sigma_z^2}{\hsig^3}+\bE_{X,Z}\left[\frac{X^2}{\hsig^3}\left(\bI\left(\Big|\frac{X}{\hsig}+Z\Big|<\beta\right)\right)\right]} \nonumber \\ &~~~~+\frac{\bE_{X}\left[-\int_{\beta-\frac{X}{\hsig}}^{\infty}(z-\beta)\phi(z)dz+\int_{-\infty}^{-\frac{X}{\hsig}-\beta}(z+\beta)\phi(z)dz\right]}{\frac{\delta\sigma_z^2}{\hsig^3}+\bE_{X,Z}\left[\frac{X^2}{\hsig^3}\left(\bI\left(\Big|\frac{X}{\hsig}+Z\Big|<\beta\right)\right)\right]},
\end{align}
where Equality (a) is according to \eqref{equ:FinalDer}. In order to simplify (\ref{equ:DerTauSigma}), we use (\ref{equ:FixedPoint1}) again
\begin{align}\label{equ:FixedPointSimp}
\delta&=\frac{\sigma_z^2 \delta}{\hsig^2}+\bE_{X,Z}\left[\left(\eta\left(\frac{X}{\hsig}+Z;\beta\right)-\frac{X}{\hsig}\right)^2\right] \nonumber \\ 
&= \frac{\sigma_z^2 \delta}{\hsig^2}+\bE_{X,Z}\left[\frac{X^2}{\hsig^2}\bI\left(\left|\frac{X}{\hsig}+Z\right|<\beta\right)\right] \nonumber \\
&~~~~+\bE_{X}\left[\int_{\beta-\frac{X}{\hsig}}^{\infty}(z-\beta)^2\phi(z)dz \right] +\bE_X \left[\int_{-\infty}^{-\beta-\frac{X}{\sigma}}(z+\beta)^2\phi(z)dz\right] \nonumber \\ 
&= \frac{\sigma_z^2 \delta}{\hsig^2}+\bE_{X,Z}\left[\frac{X^2}{\hsig^2}\bI\left(\left|\frac{X}{\hsig}+Z\right|<\beta\right)\right] \nonumber \\
&~~~~+\beta \bE_{X}\left[-\int_{\beta-\frac{X}{\hsig}}^{\infty}(z-\beta)\phi(z)dz+\int_{-\infty}^{-\frac{X}{\hsig}-\beta}(z+\beta)\phi(z)dz\right] \nonumber \\&~~~~+\bE_{X}\left[\int_{\beta-\frac{X}{\hsig}}^{\infty}z(z-\beta)\phi(z)dz+\int_{-\infty}^{-\frac{X}{\hsig}-\beta}z(z+\beta)\phi(z)dz\right].
\end{align}
Thus, we can rewrite (\ref{equ:DerTauSigma}) as
\allowdisplaybreaks
\begin{align}
\frac{d(\beta\hsig)}{d\beta}&=\frac{\delta-\bE_{X}\left[\int_{\beta-\frac{X}{\hsig}}^{\infty}z(z-\beta)\phi(z)dz+\int_{-\infty}^{-\frac{X}{\hsig}-\beta}z(z+\beta)\phi(z)dz\right]}{\frac{\delta\sigma_z^2}{\hsig^3}+\bE_{X,Z}\left[\frac{X^2}{\hsig^3}\left(\bI\left(\Big|\frac{X}{\hsig}+Z\Big|<\beta\right)\right)\right]} \nonumber \\ &=  
\frac{\delta-\bE_{X}\left[\int_{\beta-\frac{X}{\hsig}}^{\infty}z^2\phi(z)dz+\int_{-\infty}^{-\beta-\frac{X}{\hsig}}z^2\phi(z)dz\right]}{\frac{\delta\sigma_z^2}{\hsig^3}+\bE_{X,Z}\left[\frac{X^2}{\hsig^3}\left(\bI\left(\Big|\frac{X}{\hsig}+Z\Big|<\beta\right)\right)\right]} \nonumber \\ &~~~~+\frac{-\beta\bE_X\left[-\int_{\beta-\frac{X}{\hsig}}^{\infty}z\phi(z)dz+\int_{-\infty}^{-\beta-\frac{X}{\hsig}}z\phi(z)dz\right]}{\frac{\delta\sigma_z^2}{\hsig^3}+\bE_{X,Z}\left[\frac{X^2}{\hsig^3}\left(\bI\left(\Big|\frac{X}{\hsig}+Z\Big|<\beta\right)\right)\right]} \nonumber \\& = \frac{\delta-\bE_X\left[\left(\beta-\frac{X}{\hsig}\right)\phi\left(\beta-\frac{X}{\hsig}\right)+\left(\beta+\frac{X}{\hsig}\right)\phi\left(\beta+\frac{X}{\hsig}\right)\right]}{\frac{\delta\sigma_z^2}{\hsig^3}+\bE_{X,Z}\left[\frac{X^2}{\hsig^3}\left(\bI\left(\Big|\frac{X}{\hsig}+Z\Big|<\beta\right)\right)\right]} \nonumber \\ &~~~~+\frac{\bE_X\left[\int_{\beta-\frac{X}{\hsig}}^{\infty}\phi(z)dz+\int_{-\infty}^{-\beta-\frac{X}{\hsig}}\phi(z)dz\right]}{\frac{\delta\sigma_z^2}{\hsig^3}+\bE_{X,Z}\left[\frac{X^2}{\hsig^3}\left(\bI\left(\Big|\frac{X}{\hsig}+Z\Big|<\beta\right)\right)\right]} \nonumber \\&~~~~+\frac{\beta\bE_{X}\left[\phi\left(\beta-\frac{X}{\hsig}\right)+\phi\left(\beta+\frac{X}{\hsig}\right)\right]}{\frac{\delta\sigma_z^2}{\hsig^3}+\bE_{X,Z}\left[\frac{X^2}{\hsig^3}\left(\bI\left(\Big|\frac{X}{\hsig}+Z\Big|<\beta\right)\right)\right]} \nonumber \\ &=\overbrace{\frac{\delta-\bE_{X,Z}\left[\bI \left(\left|\frac{X}{\hsig}+Z\right|>\beta\right)\right]}{\frac{\delta\sigma_z^2}{\hsig^3}+\bE_{X,Z}\left[\frac{X^2}{\hsig^3}\left(\bI\left(\Big|\frac{X}{\hsig}+Z\Big|<\beta\right)\right)\right]}}^{\Delta_1}\nonumber \\ &~~~~+ \overbrace{\frac{\bE_{X}\left[\frac{X}{\hsig}\left(\phi\left(\beta-\frac{X}{\hsig}\right)-\phi\left(\beta+\frac{X}{\hsig}\right)\right)\right] }{\frac{\delta\sigma_z^2}{\hsig^3}+\bE_{X,Z}\left[\frac{X^2}{\hsig^3}\left(\bI\left(\Big|\frac{X}{\hsig}+Z\Big|<\beta\right)\right)\right]}}^{\Delta_2}.
\end{align}
$\Delta_1>0$ due to the fact that $\lambda>0$ in (\ref{equ:FixedPoint1}). Furthermore, 
\begin{itemize}
\item[-] If $X<0$ then $\phi\left(\beta-\frac{X}{\hsig}\right)<\phi\left(\beta+\frac{X}{\hsig}\right)$.
\item[-] If $X>0$ then $\phi\left(\beta-\frac{X}{\hsig}\right)>\phi\left(\beta+\frac{X}{\hsig}\right)$.
\end{itemize} 
Therefore $\Delta_2>0$ as well that means
\begin{eqnarray}\label{equ:ChainRule2}
\frac{d\lambda}{d\beta}>0.
\end{eqnarray}
Applying (\ref{equ:ChainRule1}) and (\ref{equ:ChainRule2}) into (\ref{equ:ChainRule}) completes the proof. In addition, combining Theorem \ref{thm:lassodetstate} and the fact that $\lambda$ is positive in \eqref{eq:fixedpoint21} results in  
\begin{equation}
\lim_{N \rightarrow \infty} \frac{1}{N} \sum_i \mathbb{I} \left(x^{\lambda}_i(N) \neq 0 \right) \leq \delta.
\end{equation}
\end{proof}

\subsection{Proof of Theorem \ref{thm:quasiconvex}} \label{sec:proofquasiconvex}

First, note that according to Corollary \eqref{cor:lassomsese} we have
\[
\lim_{N \rightarrow \infty} \frac{1}{N} \|\hat{x}^{\lambda}(N)- x_o \|_2^2 = \mathbb{E} _{X_o,Z}\left[(\eta(X_o + \hsig Z; \beta \hsig) - X_o)^2\right],
\]
where $(\beta, \hsig, \lambda)$ satisfy
\begin{align}\label{equ:FixedPoint}
&\hat{\sigma}^2=\sigma_z^2+\frac{1}{\delta}\bE_{X,Z}\left[(\eta(X+\hat{\sigma}Z;\hat{\sigma}\beta)-X)^2\right], \nonumber \\
& \lambda=\beta\hat{\sigma}\left(1-\frac{1}{\delta}\bP(|X+\hat{\sigma}Z|>\hat{\sigma}\beta)\right).
\end{align}
Clearly, quasi-convexity of $\mathbb{E}\left[(\eta(X+ \hsig W; \beta \hsig) - X)^2\right]$ in terms of $\lambda$ is equivalent to the quasi-convexity of $\hsig^2$ in terms of $\lambda$. Check Lemma \ref{lem:quascvx} for more information. Therefore, in the rest of the proof our goal is to prove that $\hsig^2$ is a quasi-convex function of $\lambda$. 

$\hsig^2$ is a differentiable function of $\lambda$. Therefore, according to Theorem \ref{thm:quasiconvexnessuf}, $\hsig^2$ is a quasi-convex function of $\lambda$ if and only if $\frac{d\hsig^2}{d\lambda}$ has at most one sign change (if it has a sign change it should be from negative to positive). We have
\begin{eqnarray}\label{equ:ChainRule3}
\frac{d\hsig^2}{d\lambda}=\frac{d\hsig^2}{d\beta}\frac{d\beta}{d\lambda}.
\end{eqnarray}
Hence, the rest of the proof involves two main steps:
\begin{enumerate}
\item $\frac{d \beta}{d \lambda} >0$. We have proved that this is true in Lemma \ref{lem:lambdataumon}.
\item $\frac{d\hsig^2}{d\beta}$ has exactly one sign change.  
\end{enumerate}


According to (\ref{equ:FinalDer}) we have 
\begin{align}
\frac{d\hsig^2}{d\beta}=\frac{\frac{\partial}{\partial \beta} \bE_{X,Z} \left[\left(\eta \left (\frac{X}{\hsig}+Z;\beta \right)-\frac{X}{\hsig} \right) ^2\right]}{\frac{\delta\sigma_z^2}{\hsig^3}+\bE_{X,Z}\left[\frac{X^2}{\hsig^3}\left(\bI\left(\Big|\frac{X}{\hsig}+Z\Big|<\beta\right)\right)\right]}.
\end{align}
As a result, the sign of $\frac{d\hsig^2}{d\beta}$ is the same as sign of $\frac{\partial}{\partial \beta} \bE \left[\left(\eta \left (\frac{X}{\hsig}+Z;\beta \right)-\frac{X}{\hsig} \right) ^2\right]$. However, this term is the derivative of the risk of the soft thresholding function with respect to its second parameter. This is what we proved above in Lemma \ref{lem:quasiconvexsoft}. This, combined with the fact that $d\beta/d \lambda <0$, proves that $ d \hat{\sigma}^2/ d \lambda$ has exactly one sign change.

\subsection{Proof of Lemma \ref{lem:uniquefixedpoint}}\label{sec:proofuniquefixedpoint}

 We prove this result by contradiction. Suppose that the statement of the lemma is not correct. Hence, there are at least two pairs $(\sigma_1^*, \tau_1^*)$ and $(\sigma_2^*, \tau_2^*)$ that satisfy \eqref{eq:fixedpoint:AMPfixeddet}. Assume that $\sigma_{\omega}^2 \neq 0$. It is then clear that neither $\sigma^*_1$ nor $\sigma^*_2$ are zero. Our first claim is that 
 \[
 \frac{\tau_1^*}{\sigma_1^*} = \frac{\tau_2^*}{\sigma_2^*}. 
 \]
 To prove this claim define $\beta_1 =  \frac{\tau_1^*}{\sigma_1^*} $ and $\beta_2 = \frac{\tau_2^*}{\sigma_2^*}$. We can rewrite the fixed point equations of AMP with fixed detection thresholding policy as
 \begin{eqnarray} \label{eq:fixedpoint:AMPreform}
(\sigma^{*})^2 &=& \sigma_{\omega}^2+\frac{1}{\delta} \mathbb{E}_{X, W} \left[(\eta(X + \sigma^* W; \beta^* \sigma^* ) -X)^2\right], \\
\gamma &=& \frac{1}{\delta} \mathbb{P} \left(|X+ \sigma^* W| \geq \beta^* \sigma^*\right) \label{eq:fixedpointampfd}.
\end{eqnarray}
In Lemma \ref{lem:detovertau} we proved that 
\[
\frac{d}{d \beta^*} \mathbb{P} (|X+ \sigma^* W| \geq \beta^* \sigma^*) < 0.
\]
Also, according to \eqref{eq:fixedpointampfd} the probability of detection is fixed. Considering these two facts together we conclude that $\beta_1^* = \beta^*_2$. If $\sigma_1^* = \sigma_2^*$, we already can conclude that $\tau_1^*= \tau_2^*$ that is a contradiction. Therefore, we assume $\sigma_1^* \neq \sigma_2^*$. However, this is also a contradiction, since according to Lemma \ref{lem:uniquefixedpointconc} for a fixed $\beta^*$, \eqref{eq:fixedpoint:AMPreform} has a unique fixed point. Therefore, this equation cannot have two different fixed points.

\subsection{Proof of Theorem \ref{thm:lassoampequiv}} \label{sec:prooflassoampequiv}
As we mentioned above the proof of this theorem is similar to the proof of Theorem 3.1 in \cite{BaMo11}. We only need to make a slight modification to that proof to make it work in our case. Before we proceed further, we quote the following lemma from \cite{BaMo11}.  Define
\[
C(x) = \frac{1}{2} \|y-Ax\|_2^2 + \lambda \|x\|_1. 
\]

\begin{lemma} {\rm \cite{BaMo11}} \label{lem:1BM11}
If $x,r  \in \mathbb{R}^N$ satisfy the following conditions:  
\begin{enumerate}
\item $\|r\|_2 \leq c_1 \sqrt{N}$.
\item $C(x+ r) \leq C(x)$.
\item There exists s subgradient of $C$ at point $x$, i.e.,  $sg(C,x) \in \partial C(x)$, such that $\|sg(C,x) \|_2 \leq \sqrt{N} \epsilon$;
\item Let $\nu \triangleq (1/\lambda) [A^*(y-Ax) + sg(C,x) ] \in \partial \|x\|_1$, and $S(c_2) \triangleq \{ i \in [N] \ : \ |\nu_i| \geq 1-c_2 \}$. Then, for any $S' \subset [N]$, $|S'| \leq c_3 N$, we have $\sigma_{\rm min} (A_{S(c_2) \cup S'}) \geq c_4$.  
\item The maximum and minimum non-zero singular value of $A$ satisfy $c_5^{-1} \leq \sigma_{min} (A)^2 \leq \sigma_{max} (A)^2 \leq c_5$. 
\end{enumerate}
Then $\|r\|_2 \leq \sqrt{N} \zeta(\epsilon, c_1,c_2, c_3, c_4,c_5)$. Further, for any $c_1, \ldots, c_5 >0$, $\zeta(\epsilon) \rightarrow 0$ as $\epsilon \rightarrow 0$. 
\end{lemma}

The proof of this lemma can be found in \cite{BaMo11}. We describe why this lemma plays a central role in the proof of the equivalence of AMP and LASSO. Note that later, we would like to set $x = x^t$ the estimate of the AMP algorithm for large enough $t$, and $r = \hat{x}^{\lambda} - x^t$.  Therefore, the conclusion of Lemma \ref{lem:1BM11} is that $\frac{1}{\sqrt{N}}\|x^t- \hat{x}^{\lambda}\|_2 \leq   \zeta(\epsilon, c_1,c_2, c_3, c_4,c_5)$ and will go to zero as $\epsilon \rightarrow 0$. This means that the estimate of AMP converges to the solution of LASSO. In order to prove this convergence, according to Lemma \ref{lem:1BM11}, we have to prove 5 simpler steps. Proving items 1,2, and 5 are quite standard and trivial and are left to the reader. More explanation on these parts are also mentioned in \cite{BaMo11}. The main challenges are the remaining two parts. We start with the proof of Part 3. This part ensures that subgradient of $C$ at point $x$ (estimate of AMP at iteration $t$) is small enough. 

\begin{lemma}\label{lem:thirdstep}
Let $x^t$ denote the estimate of the AMP algorithm at iteration $t$. Then there exists a subgradient $sg(C,x^t)$ of $C$ at point $x^t$ such that almost surely,
\[
\lim_{t \rightarrow \infty} \lim_{N \rightarrow \infty} \frac{1}{N} \|sg(C, x^t)\|^2 =0.
\]
\end{lemma}
\begin{proof}
Note that a similar lemma is mentioned in \cite{BaMo11}. The only difference is that to obtain $x^t$ we have used a different thresholding policy. That requires a slight change in the proof and hence we mention it here. The first step to prove this result is to construct a subgradient at $x^t$. A natural choice of subgradient  of $\ell_1$ norm at $x^t$ is given by
\begin{eqnarray*}
g_i^t = \left\{ \begin{array}{l  c}
{\rm sign}(x_{i}^t) & {\rm if} \ x_{i}^t \neq 0, \nonumber \\
\frac{1}{\tau_{t-1}} ([A^*z^{t-1}]_i + x_{i}^{t-1}) & {\rm  otherwise}. 
\end{array}
 \right.
\end{eqnarray*}
Using this subgradient for the $\ell_1$ part of $C$, we conclude that 
\[
sg(\mathcal{C}, x^t) \triangleq \lambda g^t - A^*(y- Ax^t). 
\]
 
  Note that according to Theorem \ref{thm:ampeqpseudo_lip}, $\tau_t$ converges to $\hat{\tau}$ almost surely, where $\tau^\infty$ satisfies \eqref{eq:fixedpoint:AMPfixeddet}.  We can slightly simplify this subgradient by using $y- Ax^t = z^t - \gamma z^{t-1}$. Therefore, we have
 \begin{eqnarray}
  sg(C, x^t) &=& \frac{1}{\tau_{t-1}} [\lambda \tau_{t-1} g^t - \tau_{t-1} A^* (z^t - \gamma z^{t-1})] \nonumber\\
  &=& \frac{1}{\tau_{t-1}} [\lambda \tau_{t-1} g_t - \lambda A^* z^{t-1}] - A^*(z^t - z^{t-1}) \nonumber \\
  &&+ \frac{\lambda - \tau_{t-1} (1- \gamma)}{\theta_{t-1}} A^*z^{t-1}.  
 \end{eqnarray}
The first term is clearly $\lambda (x^{t-1} - x^t)$. Therefore,
 \begin{eqnarray}\label{eq:proofampdetlasso}
\frac{1}{\sqrt{N}}  \|sg(C, x^t)\|_2 &=& \frac{\lambda}{\tau_{t-1} \sqrt{N}}\| [ x^t-x^{t-1}] \|_2- \|A^*(z^t - z^{t-1})\|_2 \nonumber\\
&&+ \|\frac{\lambda - \tau_{t-1} (1- \gamma)}{\theta_{t-1}} A^*z^{t-1}\|_2. 
 \end{eqnarray}

Since both $\lambda$ and $\sigma_{max}(A)$ are bounded, as $N \rightarrow \infty$ The first two terms converge to zero. Next, note that according to Theorem \ref{thm:ampeqpseudo_lip}, as $N \rightarrow \infty$ we have $\tau_t \rightarrow \tau^\infty$ almost surely. Hence, $\tau_t(1- \gamma)$ converges to $\tau^\infty (1-\gamma)$ almost surely, that is equal to $\lambda$. Hence the last term in \eqref{eq:proofampdetlasso} converges to zero almost surely. 
\end{proof}

The fact that the subgradient is small does not necessary mean that we are close to the optimal solution. Intuitively speaking, this is due to the fact that the function might be flat (very low curvature). Step 4 of Lemma \ref{lem:thirdstep} ensures that this does not happen. The proof of this step is exactly the same as proof of Proposition 3.7 \cite{BaMo11}. In fact, this is only about the active set of the AMP algorithm, and has been proved in a very general form in \cite{BaMo11}. Hence it covers our threshold settings as well.

\section{Simulation results}\label{sec:simulations}

\subsection{Phase transition of AMP}
In this section, we measure the empirical phase transition of AMP with fixed threshold policy and compare it with the available theoretical bound. The theoretical bound \cite{stojnic2009various,DoTa05,DoMaMo09} for phase transition under LP reconstruction is given by
\begin{align}\label{equ:Stojnic}
&\delta=\frac{\phi(z)}{\phi(z)+z\left(\Phi(z)-\frac{1}{2}\right)}, \nonumber \\
& \rho =1-\frac{z(1-\Phi(z))}{\phi(z)}.
\end{align}
On the other hand, in order to measure the empirical phase transition of AMP with fixed thresholding policy we have to set the free parameter $\gamma$. To find out this free parameter, we use the relationship between $\gamma$ and $\lambda$ defined in (\ref{eq:fixedpointlasso1}). The solution of the LASSO is given by 
\begin{align}
\hat{x}^\lambda = \arg \min_x \frac{1}{2} \|y-Ax\|_2^2+\lambda \|x\|_1.
\end{align}
As we showed in Theorem \ref{lem:activeset}, the number of elements in the active set of $\hat{x}^\lambda$, i.e. $\|\hat{x}^\lambda\|_0$, ranges between 0 and the number of measurements $n$. However, in order to measure the phase transition, we have to let $\lambda \rightarrow 0$. According to (\ref{eq:fixedpointlasso1}), this is equivalent to letting $\gamma \rightarrow 1$. As a result, we let $\gamma=1$ when calculating the empirical phase transition of the AMP algorithm with the fixed thresholding policy.
In order to calculate the empirical phase transition, we produce a heatmap in which each cell is corresponding to the success probability of recovery. We consider a $20\times 20$ matrix $\mathcal{P}$  in which columns correspond to the set $\Delta$ containing 20 equi-spaced values of $\delta$ between 0.1 and 0.9, and rows correspond to the set $\Psi$ containing 20 equi-spaced values of $\rho$ between $\rho(\delta=0.1)$ and $\rho(\delta=0.9)$. For each $\hat{\delta}\in\Delta$, we consider 50 equi-spaced values of $\rho$ in $[0.8\rho(\hat{\delta}),1.2\rho(\hat{\delta})]$, where $\rho(\hat{\delta})$ is obtained from \eqref{equ:Stojnic}, and we call the set of all these values $\Psi_{\hat{\delta}}$. For each $\hat{\delta}$ and $\tilde{\rho}\in \Psi_{\hat{\delta}}$, in order to calculate the probability of correct recovery, we use $M=20$ Monte Carlo samples. For the $j^{\text{th}}$ Monte Carlo sample, we define the success variable $S_{\hat{\delta},\tilde{\rho},j}=\mathbb{I}\left(\frac{\|\hat{x}_o-x_o\|_2}{\|x_o\|_2}< \text{tol}\right)$ where tol is the threshold of relative error for evaluating the performance of AMP, and $\hat{x}_o$ is the recovered signal by AMP using fixed thresholding policy. We then define the empirical success probability as $\hat{\mathbb{P}}_{\hat{\delta},\tilde{\rho}}=\frac{1}{M}\sum_j S_{\hat{\delta},\tilde{\rho},j}$. Having calculated $\hat{\mathbb{P}}_{\hat{\delta},\tilde{\rho}}$ for every $\tilde{\rho}\in \Psi_{\hat{\delta}}$, we use linear interpolation to fill in the entries of $\mathcal{P}$.
In running the AMP algorithm using the fixed thresholding policy, we use the following set up:
\begin{itemize}
\item[-] The size of $x_o$ is set to $N=1000$.
\item[-] The size of the measurement vector and the sparsity level are obtained according to: $n=\lfloor \delta N \rfloor$, $k=\lfloor \rho n \rfloor$.
\item[-] Measurements are noise-free and are obtained by $y=Ax_o$ where $A$ has iid elements from the Gaussian distribution $N\left(0,\frac{1}{n}\right)$.
\item[-] The threshold of the relative error is set to $\text{tol}=10^{-2}$.
\item[-] The number of iterations in AMP is set to 500.
\end{itemize}
Figure \ref{fig:HeatMap} shows the empirical heatmap of the probability of success we obtained. The black curve exhibits the theoretical curve from (\ref{equ:Stojnic}). The red and blue color refer to successful and unsuccessful recovery, respectively. We can see the coincidence of the theoretical curve and empirical phase transition in this figure.

\begin{figure}[h!]
\includegraphics[width= 12cm]{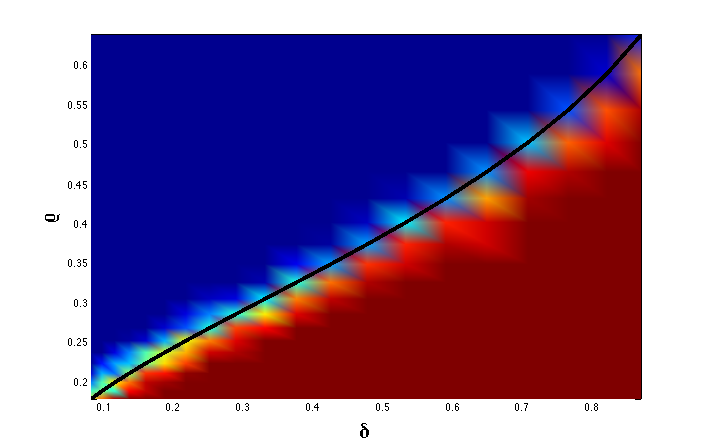}
\caption{Comparison between the empirical phase transition (heatmap of the probability of success) and theoretical phase transition curve (black curve) obtained from (\ref{equ:Stojnic}). Red color corresponds to 1 and the blue color corresponds to 0 in the heatmap. }
\label{fig:HeatMap}
\end{figure}

\subsection{Details of simulations}\label{sec:simulationdetails}

Here we include the details of the simulations whose results we reported in the previous sections. 

\subsubsection{Figure \ref{fig:activeset} }\label{subsec:fig:activeset}
The dataset we used in this simulation is taken from \cite{efron2004least}. The response variables $y \in \mathbb{R}^{442}$ is the diabetes progression in one year in 442 patients. We have 10  variables, namely age, sex, body mass index, average blood pressure, and six blood serum measurements. Therefore $x \in \mathbb{R}^{10}$.  
We have solved LASSO for different values of  $\lambda$ and presented the number of nonzero elements in $x$, i.e. $\|x\|_0$, as a function of $\lambda$ in Figure \ref{fig:activeset} . As mentioned previously, for this specific problem this function is not a decreasing.

\subsubsection{Figure \ref{fig:LassoPathRandom}}\label{subsec:fig:LassoPathRandom}
Here, we consider the recovery of a sparse signal in the presence of measurement noise. $x_o \in \mathbb{R}^{2000}$ is a sparse signal with 100 nonzero elements equal to 1.  We observe an undesampled noisy version of $x_o$, i.e.,  $y=Ax_o+w$ where $A$ is a $1000\times 2000$ matrix with iid entries having normal distribution $N(0,1)$, and $w$ is a $1000\times 1$ noise vector with iid entries having the Gaussian distribution with mean 0 and variance 0.7. We consider 100 equi-spaced values of $\lambda$ between 0 and 0.25. We then solve the optimization in (\ref{eq:LASSO}) and measure the number of $\hat{x}_{\lambda}$. The result is presented in Figure \ref{fig:LassoPathRandom}.

\subsubsection{Figures  \ref{fig:MSE}}
Simulation setting for the Figure \ref{fig:MSE} is the same as for Figure \ref{fig:LassoPathRandom}. There are only two differences. First, the entries of $w$ are obtained from the Gaussian distribution with mean 0 and variance 0.4 in one figure and with mean 0 and variance 2 in the other one. Second, the $\lambda$ values are 100 equi-spaced values between 0 and 1.

\section{Conclusions}\label{sec:con}
In this paper, we have characterized the behavior of LASSO's solution as a function of the regularization parameter $\lambda$ in the asymptotic setting  (the ambient dimension $N\rightarrow \infty$, while the ratio of the number of measurements, $n$, to the ambient dimension is fixed to $\delta$). We showed that
\begin{itemize}
\item[-] The size of the active set $\|\hat{x}_\lambda\|_0/N$ is a decreasing function of $\lambda$.
\item[-] The mean square error $\|\hat{x}_{\lambda} - x_o\|_2^2/N$ is a quasi-convex function of $\lambda$.
\end{itemize}
To demonstrate the importance of such results from the algorithmic perspective, we designed a new thresholding policy for the AMP algorithm and showed that under this thresholding policy AMP can still solve LASSO accurately.

\bibliographystyle{unsrt}
\bibliography{version10}

\end{document}